%% file: HyperHDG.tex
\documentclass[a4paper, english, 12pt, reqno, dvipsnames, final]{amsart}

\usepackage{amssymb}

\usepackage[english]{babel}
\usepackage[utf8]{inputenc}
\usepackage{tikz,tikzscale}
\tikzset{>=latex}
\usetikzlibrary{patterns}
\usepackage{enumerate,enumitem}
\usepackage{booktabs,multirow}
\usepackage{amsfonts}
\usepackage{dsfont}

\newtheorem{theorem}{Theorem}[section]
\newtheorem{lemma}[theorem]{Lemma}
\newtheorem{corollary}[theorem]{Corollary}

\theoremstyle{definition}
\newtheorem{definition}[theorem]{Definition}

\theoremstyle{remark}
\newtheorem{remark}[theorem]{Remark}

\numberwithin{equation}{section}

\let\rho\varrho

\newcommand{\setEdge}{\ensuremath{\mathcal E}}
\newcommand{\setNode}{\ensuremath{\mathcal N}}

\newcommand{\setNodeDir}{\ensuremath{\setNode_\textup D}}

\newcommand{\edge}{\ensuremath{\eta}}
\newcommand{\node}{\ensuremath{\nu}}

\newcommand{\Graph}{\ensuremath{{\mathcal G}}}
\newcommand{\SetEdge}{\ensuremath{{\mathcal E}}}
\newcommand{\SetNode}{\ensuremath{{\mathcal N}}}
\newcommand{\SetNodeBdr}{\ensuremath{\SetNode_\textup B}}
\newcommand{\SetNodeDir}{\ensuremath{\SetNode_\textup D}}

\newcommand{\SetNodeInt}{\ensuremath{\SetNode_\textup I}}

\newcommand{\Edge}{\ensuremath{E}}

\newcommand{\Node}{\ensuremath{N}}

\newcommand{\locDim}{\ensuremath{\mathfrak d}}
\let\dimEdge\locDim
\newcommand{\globDim}{\ensuremath{\mathfrak D}}
\newcommand{\trace}{\ensuremath{\gamma}}

\newcommand{\Nabla}{\ensuremath{\nabla_\Edge}}
\renewcommand{\div}{\ensuremath{\nabla\!\cdot\!}}
\newcommand{\Div}{\ensuremath{\Nabla\!\cdot\!}}

\newcommand{\Normal}{\mathbf n}

\newcommand{\jump}[1]{{[\![ #1 ]\!]}}

\newcommand{\IR}{\ensuremath{\mathbb R}}
\newcommand{\spaceH}{\ensuremath{\mathcal H}}
\newcommand{\spaceM}{\ensuremath{\mathcal M}}
\newcommand{\Hdiv}{\mathbf H^{\textup{div}}}

\newcommand{\skeletal}{\Sigma}
\newcommand{\skeletalSpace}{M}

\newcommand{\discElementSpace}{\ensuremath{V}}
\newcommand{\polynomials}{\ensuremath{\mathcal P}}

\renewcommand{\vec}[1]{\mathbf{#1}}
\DeclareMathAlphabet{\mathbfsf}{\encodingdefault}{\sfdefault}{bx}{n}

\newcommand{\dx}{\ensuremath{\, \textup d x}}

\newcommand{\ds}{\ensuremath{\, \textup d \sigma}}
\newcommand{\localU}{\ensuremath{\mathcal U}}
\newcommand{\localQ}{\ensuremath{\vec{\mathcal Q}}}

\newcommand{\ueps}{u^{\epsilon}}
\newcommand{\Eref}{E_{\mathrm{ref}}}
\newcommand{\Nref}{N_{\mathrm{ref}}}
\newcommand{\ANeps}{A_N^{\epsilon}}
\newcommand{\Aieps}{A_i^{\epsilon}}

\catcode`_=12 %
\newcommand{\code}[1]{%
  \begingroup
  \ttfamily
  \begingroup\lccode`~=`/\lowercase{\endgroup\def~}{/\discretionary{}{}{}}%
  \begingroup\lccode`~=`[\lowercase{\endgroup\def~}{[\discretionary{}{}{}}%
  \begingroup\lccode`~=`.\lowercase{\endgroup\def~}{.\discretionary{}{}{}}%
  \begingroup\lccode`~=`_\lowercase{\endgroup\def~}{_\discretionary{}{}{}}%
  \catcode`/=\active\catcode`[=\active\catcode`.=\active\catcode`_=\active
  \scantokens{#1\noexpand}%
  \endgroup
}
\catcode`_=8 %

\definecolor{mygreen}{rgb}{0,0.6,0}
\definecolor{mygray}{rgb}{0.5,0.5,0.5}
\definecolor{mymauve}{rgb}{0.58,0,0.82}
\definecolor{shade1}{gray}{.9}
\definecolor{shade2}{gray}{.75}
\definecolor{shade3}{gray}{.5}

\newlength{\flexcheckerboardsize}

\newcommand{\defineflexcheckerboard}[4]{
    \setlength{\flexcheckerboardsize}{#2}
    \pgfdeclarepatterninherentlycolored{#1}
        {\pgfpointorigin}{\pgfqpoint{2\flexcheckerboardsize}    
        {2\flexcheckerboardsize}}
        {\pgfqpoint{2\flexcheckerboardsize}
        {2\flexcheckerboardsize}}%
        {
            \pgfsetfillcolor{#4}
            \pgfpathrectangle{\pgfpointorigin}{
            \pgfqpoint{2.1\flexcheckerboardsize}    
                {2.1\flexcheckerboardsize}}
          \pgfusepath{fill}
          \pgfsetfillcolor{#3}
          \pgfpathrectangle{\pgfpointorigin}
            {\pgfqpoint{\flexcheckerboardsize}
            {\flexcheckerboardsize}}
          \pgfpathrectangle{\pgfqpoint{\flexcheckerboardsize}
            {\flexcheckerboardsize}}
            {\pgfqpoint{\flexcheckerboardsize}
            {\flexcheckerboardsize}}
            \pgfusepath{fill}
        }
}

\defineflexcheckerboard{checkerboard_s1}{1mm}{red!40}{white}

\usepackage[colorlinks = true, linkcolor = blue, citecolor = blue, urlcolor = blue]{hyperref}

\makeatletter
\newcommand\footnoteref[1]{\protected@xdef\@thefnmark{\ref{#1}}\@footnotemark}
\makeatother

\begin{document}

\title[HDG on hypergraphs]{Partial differential equations on hypergraphs and networks of surfaces: derivation and hybrid discretizations} 

\author{Andreas Rupp}
\address{School of Engineering Science, Lappeenranta--Lahti University of Technology, P.O. Box 20, 53851 Lappeenranta, Finland}
\email{andreas.rupp@fau.de}

\author{Markus Gahn}
\address{Interdisciplinary Center for Scientific Computing (IWR), Heidelberg University, Mathematikon, Im Neuenheimer Feld 205, 69120 Heidelberg, Germany}
\email{markus.gahn@iwr.uni-heidelberg.de}

\author{Guido Kanschat}
\address{Interdisciplinary Center for Scientific Computing (IWR), Heidelberg University, Mathematikon, Im Neuenheimer Feld 205, 69120 Heidelberg, Germany}
\email{kanschat@uni-heidelberg.de}


\begin{abstract}
 We introduce a general, analytical framework to express and to approximate partial differential equations (PDEs) numerically on graphs and networks of surfaces---generalized by the term hypergraphs. To this end, we consider PDEs on hypergraphs as singular limits of PDEs in networks of thin domains (such as fault planes, pipes, etc.), and we observe that (mixed) hybrid formulations offer useful tools to formulate such PDEs. Thus, our numerical framework is based on hybrid finite element methods (in particular, the class of hybrid discontinuous Galerkin methods).
 \\[1ex] \noindent \textsc{Keywords.}
 surface networks, graphs, hypergraphs, HDG, HRT, BDMH, hybrid formulation, diffusion, local conservation, continuity equation.
\end{abstract}

\date{\today}
\maketitle
\section{Introduction}
This manuscript establishes a general approach to formulate partial differential equations (PDEs) on networks of (hyper)surfaces, referred to as hypergraphs. Such PDEs consist of differential expressions with respect to all hyperedges (surfaces) and compatibility conditions on the hypernodes (joints, intersections of surfaces). These compatibility conditions ensure conservation properties (in case of continuity equations) or incorporate other properties---motivated by physical or mathematical modeling.  We illuminate how to discretize such equations numerically using hybrid discontinuous Galerkin (HDG) methods, which appear to be a natural choice, since they consist of local solvers (encoding the differential expressions on hyperedges) and a global compatibility condition (related to our hypernode conditions). We complement the physically motivated compatibility conditions by a derivation through a singular limit analysis of thinning structures yielding the same results.

Albeit many physical, sociological, engineering, and economic processes have been described by partial differential equations posed on domains which cannot be described as subsets of linear space or smooth manifolds, there is still a lack of mathematical tools and general purpose software specifically addressing the challenges arising from the discretization of these models.

Fractured porous media (see \cite{BerreDK2019} for a comprehensive review) have gained substantial attention and have become an active field of research due to their critical role with respect to flow patterns in several applications in the subsurface, in material science, and in biology. Most commonly, a fracture is described as a very thin, not necessarily planar object in which, for example Darcy's equation holds. This motivates the singular limit approximation in which a fracture is assumed to be a two dimensional \emph{surface} within the three dimensional \emph{space}. When several of these fractures meet, they form a fracture network of two-dimensional surfaces. Thus, fracture networks illustrate a physical application of the type of problem we investigate in this publication.
Moreover, a model in which the joints of two (or more) fractures are assigned additional physical properties can be found in \cite{ReichenbergerJBH2006}. Beyond this, fracture networks have been simulated using hybrid high order (HHO) methods \cite{HedinPE2019}.

Graph based models for porous media (without fractures) consist of simulating preferential flow paths within the porous matrix. One of the first publications implementing this idea is \cite{Fatt1956} who observed that a network of tubes might approximate the flow of porous media better than the classical model of tube bundles, which has also been used in the most common upscaling techniques---see \cite{SchulzRZRK2019,RayRSK2018} and the references therein for a discussion of those tube models in upscaling procedures. The tube network approach \cite{Fatt1956} has been successfully applied to couple porous media flow to free (Navier--) Stokes flow \cite{WeishauptJH2019}.

PDEs on hypergraphs are especially suitable to be used in the description of elastic networks \cite{Eremeyev2019}: Here, we discriminate between one dimensional elastic beam (rod) networks, trusses, etc.\ and two dimensional elastic plate (shell) networks \cite{LagneseL1993}, respectively. Beam networks have been used to model truss bridges and towers (most prominently the Eiffel tower) and other mechanical structures, originating the field of elastic beam theory (for instance \cite{BauchauC2009} for an introduction which also covers elastic plate models). Elastic plate models describe the stability of houses and have several engineering applications such as the description of the stability of (bend) plates (used in automobile industries and several others). They have even been used to understand interseismic surface deformation at subduction zones \cite{KandaS2010}.

Elastic beam networks have been used to evaluate elastic constants in amorphous materials. That is, the elastic properties of stiff, beam like polymers have been investigated. Such polymers are key to understanding the cytoskeleton which is an important part of biological cells \cite{Heussinger2007,Lieleg2007}, but they are also important for the healing of wounds (fibrin), for skin stability (collagen), and for the properties of paper. Moreover, such models can be used for modelling rubber \cite{PhysRevE.76.031906}, foams, and fiber networks \cite{PhysRevLett.96.017802}. 

Conservation laws in the form of PDEs on hypergraphs have been used in the simulation and optimization of gas networks \cite{RuefflerMH2018} and other networks of pipelines. They have been extended to networks of traffic (streets and data), (tele-)communication, and blood flow. For an overview of the main ideas that are related to these applications, the reader may consult \cite{Garavello2010,BressanCGHP2014}. Additionally, rigorous mathematical analysis of such problems is developing to a field of current research \cite{SkreMR2021}.

We conclude the overview over some applications by stating that regular surfaces and volumes can also be interpreted as hypergraphs. Thus, PDEs on surfaces \cite{DziukE2013} and standard ``volume" problems (in which the hyperedges have the same dimension as the surrounding space and at most two hyperedges meet in a common hypernode) are also covered by our approach.

Hypergraph models usually are approximations of problems in higher dimensional networks of thin structures, for example a network of thin pipes or thin plates in 3D. As a model example we give a rigorous  derivation of a diffusion equation on a hypergraph. More precisely, we consider a network of thin plates in three dimensions, where the thickness of the plates is small compared to their length. We denote the ratio between the thickness and the length by the small parameter $0 < \epsilon \ll 1$. Due to the different scales the computational effort for numerical simulations is very high. To overcome this problem the idea is to replace the thin-structure by a hypergraph. For this we give a rigorous mathematical justification using asymptotic analysis. We pass to the limit $\epsilon \to 0$ in the weak formulation of the problem, and derive a limit problem stated on the hypergraph. The solution of this limit-problem is an approximation of the model in the higher-dimensional thin domain. Singular limits for thin plates and shells (leading to lower-dimensional manifolds in the limit $\epsilon \to 0$) in elesticity can be found in \cite{ciarlet1997mathematical,ciarlet2000theory}. Dimension reduction for a folded elastic plate is treated in \cite{le1989folded}.
Singular limits leading to hypergraphs for fluid equations can be found in \cite{maruvsic2003rigorous}, where a Kirchhoff law in a junction of thin pipes is derived, and \cite{maruvsic2019mathematical} where junctions of thin pipes and plates are treated using the method of two-scale convergence.

The remainder of this manuscript is structured as follows: First, we discuss conservation equations on hypergraphs. Second, we rigorously formulate an elliptic model equation and investigate some of its properties in Section \ref{SEC:model_eq}. Third, we discuss its discretization by means of the HDG method in Section \ref{SEC:hdg_graph}. Fourth, we discuss how PDEs on hypergpahs can be obtained by a model reduction approach, in particular, by considering singular limits. The publication is wrapped up, by a section on possible conclusions.

\section{Conservation equations on geometric hypergraphs}
%
\subsection{Hypergraphs}
A hypergraph $\Graph=(\SetNode,\SetEdge)$ consists of a finite set $\SetEdge$ of hyperedges and a finite set $\SetNode$ of hypernodes. We refer to it as a \emph{geometric hypergraph} if the hyperedges are smooth, open manifolds of dimension $\dimEdge$ with piecewise smooth, Lipschitz boundary and the hypernodes can be identified with smooth subsets of the boundaries of these hyperedges. More specifically, the boundary of each hyperedge $\Edge_e\in \SetEdge$ is subdivided into $k_e$ nonoverlapping subsets $\Gamma^e_{i}$ such that $\partial\Edge_e = \bigcup \overline{\Gamma^e_{i}}$. We associate to $\Edge_e$ an index vector $\edge^e_1,\dots,\edge^e_{k_e}$ and isometries
\begin{gather}
 \iota^e_i\colon \Gamma^e_{i} \to \Node_{\edge^e_i}, \qquad i=1,\dots,k_e.
\end{gather}
The hypernodes are thus identified with the closures of the subsets of the boundaries of one or more hyperedges. Their dimension is $\dimEdge-1$.

$\Graph$ has the structure of a hypergraph in the classical sense as each edge $\Edge_e\in\SetEdge$ connects a set of nodes $\{\Node_{\edge^e_1},\dots,\Node_{\edge^e_{k_e}}\}\subset \SetNode$. The dual hypergraph $\Graph^*=(\SetEdge,\SetNode)$ describes the situation where each hypernode $\Node_n \in \SetNode$ connects $\ell_n$ hyperedges with indices $\node^n_1,\dots,\node^n_{\ell_n}$.

We call a hypernode $\Node_n \in \SetNode$ a boundary hypernode if $\ell_n=1$, i.e., it is part of the boundary of only a single hyperedge.
Accordingly, we define the set of boundary hypernodes $\SetNodeBdr$ and the set of interior hypernodes $\SetNodeInt = \SetNode \setminus \SetNodeBdr$.
\begin{figure}
 \includegraphics[width=.8\textwidth]{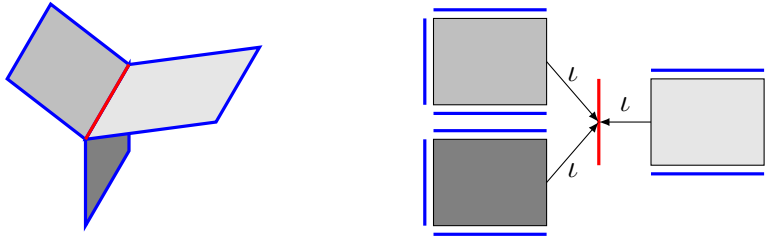}
 \caption{A hypergraph with three hyperedges of dimension 2, a hypernode (red) connecting them, and 9 boundary hypernodes (blue). Embedded hypergraph (left) in $\IR^3$ and without embedding (right). Isometries $\iota$ only shown for the interior hypernode.}\label{fig:hypergraph1}
\end{figure}

As special cases: a geometric graph is a geometric hypergraph where the edges are smooth curves and the nodes are their end points. If every hypernode is either at the boundary or connects exactly two hyperedges, the hypergraph represents a piecewise smooth manifold.

The structure might become more evident if we consider an embedded geometric hypergraph in some ambient space $\IR^\globDim$, as in Figure~\ref{fig:hypergraph1} on the left. In this case, the isometries $\iota^e_i$ are identical mappings and the hypernodes are identified with the boundary pieces $\Gamma^e_i$. On the right of this figure, the same hypergraph is displayed without embedding. In this case, the hyperedges are objects in $\IR^2$, possibly with a non-flat metric. Hypernodes are intervals in $\IR$, inheriting their metric through the isometries $\iota$.

Due to the isometries $\iota$, every point of a hypernode $\Node$ is uniquely identified with a point on the boundary of each of the hyperedges it connects. Thus, convergence of a point sequence in the union of these hyperedges to a point on the hypernode is well-defined, for instance by considering the (finitely many) subsequences on each hyperedge. Also, a distance between two points on different hyperedges sharing a hypernode is defined locally by these isometries and triangle inequality.

The domain $\Omega$ of the hypergraph, its closure, and its boundary are 
\begin{gather}\begin{split}
 \overline \Omega &= \bigcup_{\Edge\in\SetEdge} \Edge \cup \bigcup_{\Node\in\SetNode} \Node,
 \qquad \quad
 \partial\Omega = \bigcup_{\Node\in\SetNodeBdr} \Node,\\
  \Omega &= \overline \Omega \setminus \partial \Omega.
 \label{EQ:def_omega}
\end{split}
\end{gather}
In this definition, the hyperedges are considered open with respect to their topology and do not contain their boundaries.
The hypernodes are closed. We introduce the skeletal domain
\begin{gather}
    \skeletal = \bigcup_{\Node\in\SetNode} \Node.
\end{gather}
We make the assumption that $\Omega$ is connected. Note that this implies that any two hyperedges are either connected by a common node or not connected, since $\Omega$ is open, see \eqref{EQ:def_omega}.
Without such an assumption, the problems of partial differential equations below separate into subproblems, which then can be analyzed and solved independently.

In \autoref{fig:hypergraph1}, $\partial \Omega$ comprises all blue hypernodes, which also include the end points of the red hypernode. The union of the red and blue hypernodes is $\skeletal$. The domain $\Omega$ consists of the interior of the red hypernode and the the three hyperedges.

Many concepts of standard domains in $\IR^\dimEdge$ transfer to $\Omega$, even if it is not a manifold. In particular, the notion of a small open ball $B_r(x)$ with radius $r > 0$ around $x \in \Omega$, see \autoref{fig:hg_ball}, in $\Omega$ is maintained by construction and thus the notion of open subsets. A subset is called compactly embedded in $\Omega$ if its closure is contained in $\Omega$ and thus has a positive distance to $\partial\Omega$.
\begin{figure}
 \includegraphics[width=.8\textwidth]{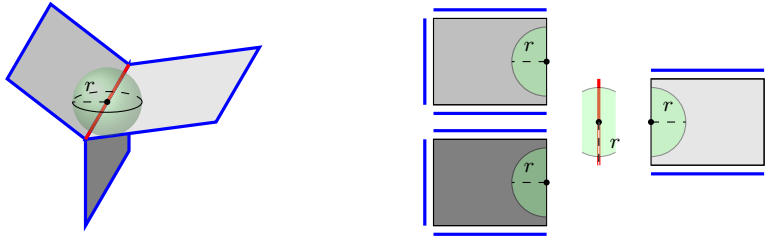}
 \caption{The hypergraph of \autoref{fig:hypergraph1} with the illustration of an open ball. The ball's center is located on the shared hypernode, and its radius is $r$.}\label{fig:hg_ball}
\end{figure}

A function is continuous on $\Omega$, if it is continuous inside each hyperedge and its limits on a hypernode are consistent between all hyperedges connected by this hypernode. Analogously, a function is in $L^2(\Omega)$ if it is in $L^2(\Edge)$ for all $\Edge \in \SetEdge$ and it is in $L^2(\Sigma)$ if it is in $L^2(\Node)$ for all $\Node \in \SetNode$.
\begin{remark}[Comparison to standard nomenclatures]
 In this article, we mix concepts from graph theory, partial differential equations, and finite elements. Thus, a clash of names was unavoidable. What is referred to as a hypernode here, is a face ---an edge in two dimensions--- in finite element literature, while the hyperedges here correspond to mesh cells or elements. In order to reduce ensuing confusion, we consistently use the term ``hyperedge''. Another difference to finite element literature is established by the fact that we consider the hypergraph fixed and are not concerned with refinement limits. Finally, we would like to point out that there has been a concept of geometric hypergraphs in the literature; it is nevertheless very limited, such that we coin this term in a new way here, meaning a hypergraph whose elements are geometric shapes themselves.
\end{remark}
\subsection{Continuity equations on hypergraphs}\label{SEC:cont_eq}
Next, we conduct a heuristic derivation, employing control volumes $V$ in the shape of infinitesimal, open hyperballs.
Let $\rho$ be a conserved quantity and $\vec J$ be its flux. Then, the conservation property of $\rho$ is usually stated in integral form such that for any such control volume $V$ there holds
\begin{gather}
    \label{EQ:conservation-1}
    \frac{d}{dt} \int_V \rho\dx = - \int_{\partial V} \vec J\cdot\Normal \ds.
\end{gather}
When $V$ is a subset of a Lipschitz manifold $\Edge$, the meaning of this statement is clear if $\vec J$ is a smooth tangential vector field in $\bar \Edge$ and $\Normal$ is the outer normal vector to $\partial V$ in the tangential plane of $\Edge$. The term $\dx$ denotes the volume element of the manifold, and $\ds$ is the induced surface element.

If the hyperball $V$ intersects a hypernode $\Node$ in which several edges meet, meaning can be given to equation~\eqref{EQ:conservation-1} by the following observation: if $\Edge_1,\dots,\Edge_\ell$ are the hyperedges which meet in $\Node$ inside $V$, then for $i\in\{1,\dots,\ell\}$ the intersection $V_i = V \cap \Edge_i$ has a piecewise smooth boundary $\partial V_i$. We observe that $\Node_V = V\cap\Node$ is in the interior of $V$ (see Figure~\ref{fig:hg_ball} for an illustration) and the boundary of $V$ is nowhere tangential to $\Node$. Thus, with the assumption that no mass is created or destroyed in the hypernode $\Node$, the conservation property~\eqref{EQ:conservation-1} can be restated as
\begin{gather}
 \frac{d}{dt} \int\limits_V \rho\dx
 = \frac{d}{dt}\sum_{i=1}^\ell \int\limits_{V_i} \rho\dx
 = - \sum_{i=1}^\ell \int\limits_{\partial V_i\setminus \Node_V} \vec J\cdot\Normal \ds.\label{EQ:conservation-2}
\end{gather}
Again, the flux $\vec J$ and the outer normal vector $\Normal$ to $V$ are well defined along $\partial V_i$ in the tangential plane of $\Edge_i$.

As a generalization of \eqref{EQ:conservation-2}, we allow for sinks and sources $f$ living in the hyperedges and $g$ living within hypernode $\Node$: This can be implemented by setting
\begin{equation}
 \frac{d}{dt} \int\limits_V \rho \dx = - \sum_{i=1}^\ell \int\limits_{\partial V_i \setminus \Node_V} \vec J\cdot\Normal \ds 
 + \int\limits_V f \dx + \int\limits_{\Node_V} g \ds,\label{EQ:balance_cond}
\end{equation}
where positive $f$ and $g$ describe sources, while negative $f$ and $g$ describes sinks.

Before we convert~\eqref{EQ:conservation-1} into a problem of partial differential equations, we make the simplifying assumption that the hyperedges and hypernodes are planar and that $\dx$ is the standard Lebesgue measure.
This way, we avoid delving into the complexities of surface partial differential equations. This simplification is purely for the ease of presentation and we refer the readers to \cite{DziukE2013} and \cite{BENARTZI2007989} for more general surfaces in the elliptic and hyperbolic settings, respectively.

Thus, in the interior of each hyperedge $\Edge$, we can apply Gauss' divergence theorem in standard form to obtain
\begin{gather}
    \int_{\partial V} \vec J \cdot\Normal\ds
    = \int_V \div \vec J \dx.
\end{gather}
If on the other hand $V$ overlaps a hypernode $\Node$ which connects hyperedges $\Edge_1\dots,\Edge_\ell$, we can still apply the divergence theorem in each hyperedge to obtain
\begin{gather}
    \int_{\partial V} \vec J \cdot\Normal\ds
    = \sum_{i=1}^{\ell} \left[
      \int_{V_i}\div \vec J \dx
      - \int_{\Node_V}
      \vec J\cdot\Normal\ds
      \right].
\end{gather}
This notion also extends to control volumes $V$ intersecting with several hypernodes in a natural way. Then, rearranging the sum over boundary integrals yields
\begin{gather}
    \label{EQ:conservation-3}
    \int_{\partial V} \vec J \cdot\Normal\ds
    = \sum_{\Edge\in\SetEdge} \int\limits_{V\cap \Edge}\div \vec J \dx
    - \sum_{\Node\in\SetNodeInt}
    \int\limits_{V \cap \Node} \jump{\vec J\cdot \Normal} \ds,
\end{gather}
for any control volume $V\subset\Omega$.
Here, $\div \vec J$ is the standard divergence of the differentiable vector field and $\jump{\cdot}$ is the summation operator such that on a hypernode $\Node$ with hyperedges $\Edge_1,\dots,\Edge_\ell$ there holds
\begin{gather}
    \label{EQ:divergence_node}
    \jump{\vec J\cdot \Normal}
    = \sum_{i=1}^{\ell} \vec J_{|\Edge_i}\cdot\Normal_{\Edge_i}.
\end{gather}
A vector field $\vec J$ is usually called solenoidal, if the left side of equation~\eqref{EQ:conservation-3} vanishes for any control volume $V$. The right hand side of this equation generalizes this notion from standard domains to hypergraphs.
Therefore, we call a piecewise smooth vector field $\vec J$ solenoidal, if
\begin{subequations}
\begin{align}
    \div \vec J(x) &= 0 && \text{ for all } x\in \Edge \text{ and } \Edge \in \SetEdge,\\
    \label{EQ:Kirchhoff-node}
    \jump{\vec J\cdot \Normal}(x) &= 0& & \text{ for all } x\in \Node \text{ and } \Node \in \SetNodeInt.
\end{align}
\end{subequations}
Note that the second condition is an extension of Kirchhoff's junction rule from points to higher dimensional hypernodes.

To put it in a nutshell, assuming there are no leaks and sources in hypernodes and hyperedges, the conservation condition~\eqref{EQ:conservation-2} induces the PDE--interface problem to find $\rho$ and $\vec J$ such that
\begin{subequations}\label{EQ:continuity}
\begin{align}
 \partial_t \rho + \div \vec J &= 0 && \text{ in all } \Edge \in \SetEdge,\\
 \jump{\vec J\cdot \Normal} &= 0& & \text{ on all } \Node \in \SetNodeInt.
\end{align}
\end{subequations}

Analogously, the continuity condition \eqref{EQ:balance_cond} induces the PDE interface problem to find $\rho$ and $\vec J$ such that
\begin{subequations}\label{EQ:balance}
\begin{align}
 \partial_t \rho + \div \vec J &= f && \text{ in all } \Edge \in \SetEdge,\\
 \jump{\vec J\cdot \Normal} &= g & & \text{ on all } \Node \in \SetNodeInt.
\end{align}
\end{subequations}

In \eqref{EQ:continuity} and \eqref{EQ:balance}, $\rho$ and $\vec J$ might be linked by some phenomenological description, i.e., $\vec J = \vec J(\rho)$ (depending on the specific application). Both equations are complemented by appropriate initial and boundary conditions. Beyond this, additional continuity constraints might be formulated, such as $\rho \in C(\Omega)$, $\rho \in C^\infty(\bigcup \Edge)$, \ldots.
\begin{remark}\label{REM:hybrid_formulation}
 The interface problems in \eqref{EQ:continuity} and \eqref{EQ:balance} resemble the \emph{hybrid} or \emph{hybridized} formulation of a PDE, which was introduced for instance in context of the mixed elements of Raviart--Thomas and Brezzi--Douglas--Marini in \cite{RaviartT1977a,RaviartT1977b,BrezziDM1985}.
\end{remark}
We are, exemplary, going to discuss the continuity equation~\eqref{EQ:continuity} in the context of diffusion 
problems in section \ref{SEC:model_eq}.
%
\section{Elliptic model equation}\label{SEC:model_eq}
%
The standard diffusion equation in mixed form defined on a hypergraph $\Graph = (\SetEdge, \SetNode)$ is a conservation equation of type~\eqref{EQ:continuity} for the flux $\vec J = -\kappa\nabla u$ of a scalar function $u$. This is for instance known as Fourier's law of thermal conduction, where $u$ is the temperature and $\kappa$ is the dimensionless heat conductivity of the material. It is also Fick's law of diffusion where $u$ is a concentration and $\kappa$ is the diffusion coefficient.

Like in the previous section, we simplify the presentation by assuming that all hyperedges are flat and thus can be identified with a domain in $\IR^\dimEdge$. In the more general case, the differential operators must be replaced by their differential geometric counterparts as in \cite{DziukE2013}.

We focus on the stationary case and set the time derivative in~\eqref{EQ:continuity} to zero. Thus, the discussion of the previous section leads to the following problem: find $u$ satisfying
\begin{subequations}\label{EQ:diffusion_primal}
\begin{align}
 - \div (\kappa \nabla u ) & = f && \text{ in all } \Edge\in\SetEdge, \label{EQ:diffusion_eq}\\
 u & = u_\textup D && \text{ on all } \Node \in \SetNodeDir, \label{EQ:Dirichlet}\\
 u_{|\Edge_1} & = u_{|\Edge_2} && \text{ on all } \Node \in \SetNode,\; \Node \subset \partial \Edge_1 \cap \partial \Edge_2,\label{EQ:diffusion_cont}\\
 - \jump{\kappa \nabla u \cdot \Normal} & = g && \text{ on all } \Node \in \SetNode \setminus \SetNodeDir, \label{EQ:interface}
\end{align}
\end{subequations}
for all $\Edge_1, \Edge_2 \in \SetEdge$, right hand sides $f$ and $g$, and a diffusion coefficient $\kappa \ge \kappa_0 > 0$. A justification by taking the limit of thin domains can be found in Section~\ref{SEC:limit} below.

We observe that in~\eqref{EQ:diffusion_primal} the diffusion equation \eqref{EQ:diffusion_eq} is complemented by three  boundary and interface conditions.
First, it is closed by a ``Dirichlet'' boundary condition~\eqref{EQ:Dirichlet}: We choose a non-empty set $\SetNodeDir \subset \SetNode$ of ``Dirichlet'' hypernodes, on which we impose $u = u_\textup D$ for a prescribed boundary value $u_D$.
In~\eqref{EQ:diffusion_cont}, we employ a continuity constraint. This constraint prohibits jumps in the primary unknown across interior nodes, and therefore, loosely speaking, imitates the standard constraint that $u \in H^1$ of the domain.

On interior nodes $\Node \in \SetNodeInt \subset \SetNode \setminus \SetNodeDir$, we set out with Kirchhoff's junction law, but with the option of a concentrated source $g$ in~\eqref{EQ:interface}. This equation also incorporates the Neumann condition $\vec -\kappa \nabla u \cdot \Normal = g$, since on a boundary hypernode the sum in the definition~\eqref{EQ:divergence_node} of the operator $\jump{\cdot}$ reduces to a single hyperedge. Note that \eqref{EQ:interface} for $g=0$ on interior nodes serves as compatibility condition for mimicking $-\kappa \nabla u \in H^\textup{div}$.
\begin{definition}[Function spaces on hypergraphs]\label{DEF:h}
 For each $\Edge\in\SetEdge$ let $H^1(\Edge)$ be the standard Sobolev space on $\Edge$ and $\gamma\colon H^1(\Edge)\to H^{1/2}(\partial\Edge)$ be the standard trace operator.
 
 Then, we define
 \begin{equation}
  \spaceH = \left\{ u \in \bigoplus_{\Edge \in \SetEdge} H^1(\Edge) \;\middle|\; \begin{array}{@{\,}c@{\, }} \gamma_1 u = \gamma_2u \\ \text{on } \Node = \partial \Edge_1 \cap \partial \Edge_2, \; \Node \in \SetNode \end{array} \right\},
 \end{equation}
 where $\gamma_1u$ and $\gamma_2u$ are the traces of $u$ from the hyperedges $\Edge_1$ and $\Edge_2$ on $\Node$, respectively.
 Due to the equality of traces in the definition of $\spaceH$, we can define the trace operator to the skeleton
 \begin{equation}
  \trace \colon \spaceH \to \spaceM := \left\{ \mu \in L^2(\skeletal) \;\middle|\; \begin{array}{@{\,}c@{\, }} \mu_{|\partial \Edge} \in H^{1/2}(\partial \Edge) \\ \text{for all } \Edge \in \setEdge \end{array} \right\}.
 \end{equation}
 
 Additionally, the spaces $\spaceH_0$ and $\spaceM_0$ are defined as
 \begin{align}
  \spaceM_0 := & \{ \mu \in \spaceM \colon \mu_{|\Node} = 0 \text{ for all } \Node \in \SetNodeDir \},\\
  \spaceH_0 := & \{ u \in \spaceH \colon \trace u \in \spaceM_0 \},
 \end{align}
 and we denote the  dual spaces of $\spaceH_0$ by $\spaceH_0^\star$ and of $\spaceM_0$ by $\spaceM_0^\star$.
 
 Norms ($\|\cdot\|_\spaceH$ and $\|\cdot\|_\spaceM$) on the respective spaces ($\spaceH$ and $\spaceM$) are induced by summed versions of the local scalar-products:
 \begin{align}
  ( u, v )_\spaceH := & \sum_{\Edge \in \SetEdge} (u_{|\Edge}, v_{|\Edge})_{H^1(\Edge)}, &
  \| u \|^2_\spaceH := & (u, u)_\spaceH, \\
 \langle \lambda, \mu \rangle_\spaceM := & \sum_{\Edge \in \SetEdge} \langle \lambda_{|\Edge}, \mu_{|\Edge}\rangle_{H^{1/2}(\partial \Edge)}, &
 \| \mu \|^2_\spaceM := & \langle \mu, \mu \rangle_\spaceM,
 \end{align}
 such that
 \begin{gather}
     \| u \|^2_\spaceH = \sum_{\Edge \in \SetEdge} \| u_{|\Edge} \|^2_{H^1(\Edge)}
     \quad\text{and}\quad
     \| \mu \|^2_\spaceM = \sum_{\Edge \in \SetEdge} \| \mu_{|\partial \Edge} \|^2_{H^{1/2}(\partial \Edge)}.
 \end{gather}
\end{definition}
These definitions have a few immediate consequences:
\begin{enumerate}
 \item $\trace\colon \spaceH \to \spaceM$ is a well-defined and surjective, linear, and continuous operator.
 \item We have the Gelfand triple relations
 \begin{gather}
  \spaceH_0 \hookrightarrow L^2(\Omega) \cong [L^2(\Omega)]^\star \hookrightarrow \spaceH_0^\star,\\
  \spaceM_0 \hookrightarrow L^2(\Sigma) \cong [L^2(\Sigma)]^\star \hookrightarrow \spaceM_0^\star.
 \end{gather}
\end{enumerate}
 Note that $\spaceM^\star$ is analogous to space $M$ in of Raviart and Thomas \cite{RaviartT1977a}.

\begin{lemma}
 The space $\spaceH$ with inner product $(\cdot, \cdot)_\spaceH$ is a Hilbert space.
\end{lemma}
\begin{proof}
 Obviously, $\spaceH$ is a subspace of the Hilbert space $\bigoplus H^1(\Edge)$, and the function
 \begin{equation*}
  h \colon \bigoplus_{\Edge \in \SetEdge} H^1(\Edge) \ni u \mapsto \sum_{\Node \in \SetNodeInt} \sum_{\bar \Edge_1,\bar \Edge_2 \supset \Node} \| u|_{\Edge_1} - u|_{\Edge_2} \|^2_{L^2(\Node)} \in \IR
 \end{equation*}
 is continuous and $\spaceH$ is its kernel. Thus, $\spaceH$ is closed.
\end{proof}
\begin{definition}
\label{DEF:weak_sol}
 A weak solution to the primal formulation of \eqref{EQ:diffusion_primal} with $\kappa \in L^\infty(\Omega)$, $f \in \spaceH_0^\star$, and $g \in \spaceM_0^\star$ is a function $u \in \spaceH$ with $\gamma u = u_\textup D$ on all $\Node \in \SetNodeDir$, and
 \begin{equation}\label{EQ:weak_diff}
  \sum_{\Edge \in \SetEdge} \int_\Edge \kappa \nabla u \cdot \nabla v \dx = \langle f, v\rangle_{\spaceH_0^\star, \spaceH_0}  - \langle g, v \rangle_{\spaceM_0^\star, \spaceM_0} \qquad \forall v \in \spaceH_0.
 \end{equation}
\end{definition}
In particular, if $f \in L^2(\Omega)$ and $g \in L^2(\Sigma)$, we can rewrite \eqref{EQ:weak_diff} as
\begin{equation}
  \sum_{\Edge \in \SetEdge} \int_\Edge \kappa \nabla u \cdot \nabla v \dx = \int_\Omega f v \dx  - \sum_{\Node \in \SetNode} \int_\Node g v \ds \qquad \forall v \in \spaceH_0.
\end{equation}
%
\subsection{Existence and uniqueness of solutions}
\label{SEC:ExUnique}
%
\begin{theorem}\label{TH:ex_uni}
 Assume for $u_\textup D$ that there is a lifting $\bar u_\textup D \in \spaceH$ with $\bar u_\textup D = u_\textup D$ on all $\Node \in \SetNodeDir$. If $\kappa \in L^{\infty}(\Omega)$ with $\kappa \ge \kappa_0 > 0$ a.~e., $f \in \spaceH_0^\star$, $g \in \spaceM_0^\star$, and all $\Edge \in \SetEdge$ are Lipschitz domains, there is an unique weak solution $u$ according to Definition~\ref{DEF:weak_sol}, which continuously depends on the data.
\end{theorem}
\begin{proof}
Due to the existence of $\bar u_\textup D$, we can reduce the problem to the one with homogeneous Dirichlet values if we replace $u$ by $u-\bar u_\textup D$ and modifying the right hand side accordingly. Since the right hand side is bounded and $\spaceH_0$ is a Hilbert space, it suffices to show ellipticity of the weak form to conclude the proof by the Lax--Milgram lemma. We note that for $v\in \spaceH_0$ there holds
\begin{gather}
    \sum_{\Edge \in \SetEdge} \int_\Edge \kappa \nabla v \cdot \nabla v \dx 
    \ge \kappa_0 \sum_{\Edge \in\SetEdge} \| \nabla v \|^2_{L^2(\Edge)}
\end{gather}
Thus, the following Poincaré--Friedrichs inequality implies ellipticity and concludes the proof.
\end{proof}

\begin{lemma}[Poincaré--Friedrichs inequality for $\spaceH_0$]\label{LEM:Poincare1}
For all $v \in \spaceH_0$ it holds that
\begin{align*}
 \Vert v \Vert_{L^2(\Omega)} \le C  \sum_{\Edge \in \SetEdge} \| \nabla v \|_{L^2(\Edge)}.
\end{align*}
 \end{lemma}
 \begin{proof}
 Similar to the standard case of subdomains in $\IR^{\locDim}$,
 this inequality follows easily by contradiction:
 To this end, we assume that there is a sequence $(v_n)_{n=1,\ldots} \subset \spaceH_0$ with
 \begin{equation}
  \lVert v_n \rVert_{L^2(\Omega)} =1
  \qquad\text{and}\qquad
  \sum_{\Edge \in \SetEdge} \lVert \nabla v_n \rVert_{L^2(\Edge)} \le \frac1n.
 \end{equation}
 Thus, $v_{n|\Edge}$ is bounded in $H^1(\Edge)$ for all $\Edge \in \SetEdge$. Hence, by the weak compactness of the unit ball in $H^1(E)$ and the Rellich-Kondrachov theorem, there exists a subsequence (also denoted $v_n$) such that
 \begin{equation}
  v_{n|\Edge} \rightarrow \tilde v_{|\Edge} \text{ in } {L^2(\Edge)},
  \qquad \text{and} \qquad
  v_{n|\Edge} \rightharpoonup \tilde v_{|\Edge} \text{ in } H^1(\Edge).
 \end{equation}
 We have that $\tilde v \in \spaceH_0$ (due to its completeness), and that the seminorm $\sum_{\Edge \in \SetEdge} \|\nabla \tilde v\|_{L^2(\Edge)} = 0$. Thus, $\tilde v$ is constant in all $\Edge \in \SetEdge$ and overall continuous. Therefore it is overall constant and has to be zero, due to the zero boundary condition on Dirichlet nodes and the connectedness of $\Omega$. Hence, the strong convergence of $v_{n|E}$ in $L^2(E)$ implies $\lVert v_{n|E} \rVert_{L^2(E)} \rightarrow 0$, which contradicts $\| \tilde v \|_{L^2(\Omega)} = 1$. Therefore, the  Poincaré--Friedrichs inequality is valid.
\end{proof}
%
\section{HDG method for elliptic model equation}\label{SEC:hdg_graph}
%
When we derived PDE problems on hypergraphs, we were led to a formulation local on each hyperedge with coupling conditions on hypernodes. This is a structure which is nicely reflected in hybridized methods. Indeed, there the separation goes one step further. By putting degrees of freedom on the hypernode, values on hyperedges are not coupling anymore to other hyperedges across these hypernodes, but only to the values on the hypernodes constituting their boundary. Thus, differing from standard or discontinuous finite element methods, the number of hyperedges attached to a hypernode does not affect the solution process on a single hyperedge. Therefore, we consider hybridized methods ideally suited to PDEs on hypergraphs.

Hybridized discontinuous Galerkin (HDG) methods break the continuity condition~\eqref{EQ:diffusion_cont} by introducing Lagrange multipliers on each hypernode which enforce the continuity of fluxes~\eqref{EQ:interface} weakly.
It turns out though, that the Lagrange multiplier is an approximation to the solution $u$ of~\eqref{EQ:diffusion_primal} on the skeleton itself.

With such methods, the actual PDE \eqref{EQ:diffusion_eq} is represented locally on each hyperedge by Steklov-Poincaré operators on the hyperedges, which transform function values to flux values on the boundary of the hyperedges, a process called ``local solver'' in HDG terminology.
The global problem is posed in terms of the degrees of freedom on the hypernodes only, yielding a square, linear system of equations.

In this respect, HDG methods have a similar structure as the family of HHO methods. These are based on hybridizing the primal formulation and lead to a rather simple error analysis on polytopic meshes where only $L^2$ projections are used (as opposed to the rather complicated projections used for HDG). This is achieved by a novel stabilization design \cite{DiPietro2015}. For recent developments in hybrid high-order and HDG methods, the reader may consult \cite{Burman2018,Qiu2016}.

The separation of the local solution of bulk problems from the global coupling of interface variables is also achieved by the virtual element method \cite{VEM0,VEM1}. Thus, it fits into our view of coupled differential equations on connected hyperedges. Different to the methods discussed so far, it does not rely on polynomial shape functions inside mesh cells but rather on forms of fundamental solutions of any shape \cite{VEM2}. Accordingly, when applied to hypergraphs, the actual type of local solvers and of the specific boundary trace operators will differ from our approach, but remain within the same principal concept.
\subsection{The hybridized dual mixed formulation}\label{SEC:hdg_mixed}
In physical applications, there often is a need to receive reasonable approximations for both the primal unknown $u$ and the dual unknown $\vec q = - \kappa \nabla u$. In other words, considering diffusion, we would like to know both the distribution of some species' concentration and the species ``movement" (flux).  This becomes particularly important if we interpret $u$ as pressure and $\vec q$ as fluid flow through a porous medium (Darcy's equation). In this situation, the flow field $\vec q$ will govern the movement of chemical species dissolved within the fluid. It is often the main quantity of interest and conservativity is crucial. Therefore, we turn to the mixed formulation.

The mixed HDG methods use the weak, dual, mixed, hybrid formulation of \eqref{EQ:diffusion_primal}, i.e., find $(u,\vec q, \lambda) \in L^2(\Omega) \times \bigoplus \Hdiv(\Edge) \times \spaceM$ with $\lambda = u_\textup D$ on all $\Node \in \SetNodeDir$ such that
\begin{subequations}\label{EQ:diffusion_mixed}
\begin{align}
 \int_\Edge \left[ u (\div \vec p ) - \kappa^{-1} \vec q \cdot \vec p \right] \dx & = \int_{\partial \Edge} \lambda \vec p \cdot \Normal \ds && \forall \vec p \in \bigoplus_{\Edge\in\SetEdge} \Hdiv(\Edge),
 \label{EQ:primal_dual}\\
  \int_\Edge v\div \vec q \dx & = \int_\Edge f v \dx && \forall v \in L^2(\Omega),
  \label{EQ:div_is_f}\\
  \sum_{\Edge \in \SetEdge} \int_{\partial\Edge} (\vec q \cdot \Normal) \mu \ds &= \langle g, \mu \rangle_{\spaceM^\star_0, \spaceM_0} && \forall \mu \in \spaceM_0.\label{EQ:interface_mixed}
\end{align}
\end{subequations}

Well-posedness of this formulation can be deduced from Theorem \ref{TH:ex_uni} if $f \in L^2(\Omega)$. Indeed, on the one hand, this implies that the (uniquely existing) solution $\overline{u}$ of Definition \ref{DEF:weak_sol} solves \eqref{EQ:diffusion_mixed} with $u = \overline{u}$, $\vec q = - \kappa \nabla \overline{u}$, and $\lambda = \gamma \overline{u}$. On the other hand, for any solution $(u,\vec q, \lambda) \in L^2(\Omega) \times \bigoplus \Hdiv(\Edge) \times \spaceM$ of \eqref{EQ:diffusion_mixed}, we have $u \in \spaceH$ (by the space's definition), and $\vec q = -\kappa \nabla u$ in the weak sense. Therefore, any solution to \eqref{EQ:diffusion_mixed} satisfies Definition \ref{DEF:weak_sol}.

Equations \eqref{EQ:primal_dual} \& \eqref{EQ:div_is_f} are local equations on the hyperedge, like in the standard case of a domain. They only couple to the Lagrange multipliers on the boundary of the hyperedge. Thus, we can eliminate them locally in the fashion of the Schur complement method.
To this effect, we introduce the local solution operator $S_\Edge\colon \spaceM\to\spaceM^*$ for the right hand side $f=0$. It is in fact a Steklov-Poincaré operator on $\Edge$ mapping the Dirichlet data $\lambda$ to the normal trace of the flux in~\eqref{EQ:interface_mixed}.
  
Then, the solution $\lambda$ of~\eqref{EQ:diffusion_mixed} can be characterized as
\begin{gather}
  \label{EQ:equation-steklov}
      \sum_\Edge \langle S_\Edge\lambda, \mu \rangle_{\spaceM^\star_0, \spaceM_0}
      =
      \langle g, \mu \rangle_{\spaceM^\star_0, \spaceM_0}.
\end{gather}

\begin{remark}
  The Steklov-Poincaré operators $S_\Edge$ in this equation are the same ones as in the case of a manifold. They do not depend on the connectivity of a hypernode to other hyperedges. Thus, their implementation does not differ from that of a standard finite element method. The  only difference lies in the structure of the sum on the left, and is thus almost purely of algebraic nature.
\end{remark}

For inhomogeneous right hand side $f\neq 0$, we can define the operators
\begin{gather}
    \begin{split}
    (\tilde \localU, \tilde \localQ) \colon L^2(\Omega) &\to L^2(\Omega) \times \bigoplus_{\Edge \in \SetEdge} \Hdiv(\Edge)\\
    f &\mapsto (u,\vec q)
    \end{split}
\end{gather}
by an edge-wise solution of \eqref{EQ:primal_dual} and \eqref{EQ:div_is_f} with $\lambda \equiv 0$.

In order to obtain a better understanding of the Steklov-Poincaré operators, we follow the route laid out in \cite{CockburnGL2009} for the discrete version and define the solution operators
\begin{gather}
    \begin{split}
    (\localU, \localQ) \colon \spaceM &\to L^2(\Omega) \times \bigoplus_{\Edge \in \SetEdge} \Hdiv(\Edge)\\
    \lambda &\mapsto (u,\vec q)
    \end{split}
\end{gather}
which map a given $\lambda$ to the element-wise solution of \eqref{EQ:primal_dual} and \eqref{EQ:div_is_f} with $f \equiv 0$.

The well-posedness and linearity of all local solution operators follow directly from the fact (see for instance \cite{BoffiBrezziFortin13}) that the mixed formulation on a single hyperedge is well-posed for any given $\lambda \in H^{1/2}(\partial \Edge)$ and its solution depends continuously on $\lambda$. 
Entering these solution operators into~\eqref{EQ:interface_mixed} yields
\begin{equation}
 - \sum_{\Edge \in \SetEdge} \int_{\partial\Edge} (\localQ\lambda \cdot \Normal) \mu \ds = \sum_{\Edge \in \SetEdge} \int_{\partial\Edge} (\tilde \localQ f \cdot \Normal) \mu \ds - \langle g, \mu \rangle_{\spaceM^\star_0, \spaceM_0},
\end{equation}
since $\vec q = \localQ \lambda + \tilde \localQ f$.
By some simple transformations of~\eqref{EQ:diffusion_mixed}, we can write~\eqref{EQ:equation-steklov} with inhomogeneous right hand sides in terms of bilinear and linear forms. This argument allows to reduce the problem to finding $\lambda \in \spaceM$ with $\lambda = u_\textup D$ on all $\Node \in \SetNodeDir$, such that
\begin{subequations}
\begin{equation}
 a(\lambda, \mu) = b(\mu) \qquad \forall \mu \in \spaceM_0,
\end{equation}
\begin{align}
 a(\lambda, \mu ) & = \sum_{\Edge \in \SetEdge} \int_\Edge \kappa^{-1} \localQ \lambda \cdot \localQ \mu \dx,\label{EQ:cond_bil}\\
 b(\mu) & = \sum_{\Edge \in \SetEdge} \int_{\partial\Edge} (\tilde \localQ f \cdot \Normal) \mu \ds - \langle g, \mu \rangle_{\spaceM^\star_0, \spaceM_0}.
\end{align}
\end{subequations}
Obviously, bilinear form $a$ and linear form $b$ are continuous due to the continuity of operators $\localQ$ and $\tilde \localQ$. Surprisingly, we have recovered a symmetric bilinear form. The following lemma is a key to the discrete well-posedness and adds the fact that this form is even $\spaceM_0$-elliptic.
\begin{lemma}
 If $\kappa \ge \kappa_0 > 0$ and $\kappa \in L^\infty(\Omega)$, bilinear form $a$ from \eqref{EQ:cond_bil} is $\spaceM_0$ elliptic.
\end{lemma}
\begin{proof}
Like in the proof of the Poincaré-Friedrichs inequality, we prove ellipticity of $a(\cdot,\cdot)$ by a contradiction argument. To this end, let $\lambda_n$ a sequence in $\spaceM_0$ such that
\begin{gather}\label{EQ:contradiction-poincare2}
    \lVert\lambda_n\rVert_{\spaceM} = 1
    \qquad\text{and}\qquad
    \lVert\localQ\lambda_n\rVert_{L^2(\Omega)} \to 0.
\end{gather}
Thus, there exists a subsequence $\lambda_n \rightharpoonup \tilde\lambda$ in $\spaceM$, and by the compact embedding of $H^{1/2}(\partial\Edge)$ in $L^2(\partial\Edge)$ there holds again for a subsequence $\lambda_n \to \tilde \lambda$ in $L^2(\skeletal)$. Since $\localQ$ is continuous and $\lambda_n$ converges weakly in $H^{\frac12}(\partial E)$ we obtain $\localQ \lambda_n \rightharpoonup Q\tilde\lambda$ weakly in $\Hdiv(\Edge)$ for each hyperdege and therefore also in $L^2(\Omega)$. Thus,~\eqref{EQ:contradiction-poincare2} implies that $\localQ\tilde\lambda = 0$. We denote by $(u,\vec q)$ the solution to~\eqref{EQ:primal_dual} and~\eqref{EQ:div_is_f} associated to $\tilde\lambda$, especially we have $\vec q = \localQ \tilde\lambda = 0$. Hence, for every 
$\Edge \in \SetEdge$ we have
\begin{equation}
 \int_\Edge u (\div \vec p) \dx = \int_{\partial \Edge} \lambda \vec p \cdot \Normal \ds \qquad \forall \vec p \in \Hdiv(\Edge).\label{EQ:const_u}
\end{equation}
Moreover, we have that the divergence operator
\begin{equation}
 \mathrm{div}\colon H^1_0(\Edge)^\locDim \to L^2_0(\Edge) := \{ u \in L^2(\Edge)\colon \int_\Edge u = 0 \}
\end{equation} is surjective, and therefore
\begin{equation}
 \int_\Edge u \varphi \dx = 0 \qquad \forall \varphi \in L^2_0(\Edge).
\end{equation}
This, however, implies that $u$ is constant on $\Edge$. Furthermore, since $u$ is constant on $\Edge$, we can deduce by \eqref{EQ:const_u} and Gauss' divergence theorem that $\tilde\lambda = u$ is constant and
\begin{gather}
    \label{EQ:constant-poincare2}
    \tilde\lambda \equiv \operatorname{const} \quad\text{on}\;\partial \Edge
    \qquad \forall \Edge \in \SetEdge.
\end{gather}
The contradiction argument is concluded by the fact that some hyperedges are adjacent to the Dirichlet nodes and thus  $\tilde\lambda = 0$ on their boundary. For the other hyperedges, $\tilde\lambda = 0$ follows from connectedness of $\Omega$. Thus, $\tilde\lambda\equiv 0$ on $\skeletal$ in contradiction to $\lVert\lambda\rVert_{\spaceM} = 1$. Altogether we showed that for all $\lambda \in \spaceM_0$ it holds that
\begin{align*}
    \lVert \lambda \rVert_{\spaceM} \lesssim \lVert \localQ \lambda \rVert_{L^2(\Omega)}.
\end{align*}
Together with \eqref{EQ:cond_bil} we obtain for a positive constant $\alpha$
\begin{align*}
    a(\lambda,\lambda) \geq \alpha \lVert \lambda \rVert_{\spaceM}^2,
\end{align*}
i.e., the ellipticity of $a$ on $\spaceM_0$.

Thus, the Lax-Milgram lemma guarantees a unique solution as soon as it is clear that $f$ and $g$ generate a right hand side in the dual of $\spaceM_0$. But for $g$ this is obvious as $\spaceM_0\subset L^2(\skeletal)$. For $f\in L^2(\Omega)$, there is a unique solution of~\eqref{EQ:primal_dual} and~\eqref{EQ:div_is_f} with $\lambda=0$. Its trace $\vec q\cdot\Normal$ is in $H^{-1/2}(\partial\Edge)$ for each hyperedge and thus bounded on $\spaceM$.
\end{proof}
\subsection{HDG methods in dual mixed form}\label{SEC:HDG_form}
Let $\hat M$ be some finite dimensional, scalar function space. Then, we define the space of discrete functions on the skeleton $\skeletal$ by
\begin{gather}\label{EQ:skeletal_space}
 \skeletalSpace := \left\{ \lambda \in L^2(\skeletal) \;\middle|\;
 \begin{array}{r@{\,}c@{\,}ll}
  \lambda_{|\Node} &\in& \hat M & \forall \Node \in \SetNode\\
  \lambda_{|\Node} &=& 0 & \forall \Node \in \SetNodeDir    
 \end{array}
 \right\}.
\end{gather}
The mixed HDG methods involve a local solver on each hyperedge $\Edge \in \SetEdge$, producing hyperedge-wise approximations $U_\Edge \in V_\Edge$ and and $\vec Q_\Edge \in \vec W_\Edge$ of the functions $u$ and $\vec q$ in equation \eqref{EQ:diffusion_mixed}, respectively. Here, $V_\Edge$ is some finite dimensional, scalar function space, and $\vec W_\Edge$ is some finite dimensional, vector valued function space. We will also use the concatenations of the spaces $V_\Edge$ and $\vec W_\Edge$, respectively, as a function space on $\Omega$, namely
\begin{gather}\label{EQ:dg_spaces}
 \begin{aligned}
  \discElementSpace &:=\bigl\{ v \in L^2(\Omega) & \big|\;v_{|\Edge} &\in V_\Edge, &\forall \Edge &\in \SetEdge \bigr\},\\
  \vec W &:=\bigl\{ \vec q \in L^2(\Omega;\IR^d) & \big|\;\vec q_{|\Edge} &\in \vec W_\Edge, &\forall \Edge &\in \SetEdge \bigr\}.
 \end{aligned}
\end{gather}
The HDG scheme for \eqref{EQ:diffusion_mixed} on a hypergraph $\Graph$ consists of the local solver and a global coupling equation. The local solver is defined hyperedge-wise by a weak formulation of \eqref{EQ:diffusion_mixed} in the discrete spaces $V_\Edge \times \vec W_\Edge$ and defining suitable numerical traces and fluxes. Namely, given $\lambda \in \skeletalSpace$ find $U_\Edge \in V_\Edge$ and $\vec Q_\Edge \in \vec W_\Edge$ , such that
\begin{subequations}\label{EQ:hdg_scheme}
 \begin{align}
  \int_\Edge \frac{1}{\kappa} \vec Q_\Edge \cdot \vec p \dx - \int_\Edge U_\Edge \Div \vec p \dx & = - \int_{\partial \Edge} \lambda \vec p \cdot \Normal \ds \label{EQ:hdg_primary}\\
  \int_{\partial \Edge} ( \vec Q_\Edge \cdot \Normal + \tau  U_\Edge ) v \ds - \int_\Edge \vec Q_\Edge \cdot \Nabla v \dx
  & = \tau \int_{\partial \Edge} \lambda v \ds \label{EQ:hdg_flux}
 \end{align}
\end{subequations}
hold for all $v \in V_\Edge$, and all $\vec p \in \vec W_\Edge$, and for all $\Edge \in \SetEdge$. Here, $\tau \ge 0$ is the penalty coefficient. While the local solvers are implemented hyperedge by hyperedge, it is helpful for the analysis to combine them by concatenation. Thus, the local solvers define a mapping
\begin{gather}
 \begin{split}
  \skeletalSpace & \to \discElementSpace \times \vec W\\
  \lambda &\mapsto (\localU \lambda, \localQ \lambda),
 \end{split}
\end{gather}
where for each hyperedge $\Edge \in \SetEdge$ holds $\localU \lambda = U_\Edge$ and $\localQ \lambda = \vec Q_\Edge$. Analogously, we set  $\localU (f, u_\textup D)$ and $ \localQ (f, u_\textup D)$, where now the local solutions are defined by the system
\begin{subequations}\label{EQ:hdg_f}
 \begin{align}
  \int_\Edge \frac{1}{\kappa} \vec Q_\Edge \cdot \vec p \dx - \int_\Edge U_\Edge \Div \vec p \dx & = - \int_{\partial \Edge} u_\textup D \vec p \cdot \Normal \ds\label{EQ:hdg_f_primary}\\
  \int_{\partial \Edge} ( \vec Q_\Edge \cdot \Normal + \tau U_\Edge ) v \ds - \int_\Edge \vec Q_\Edge \cdot \Nabla v & \dx \label{EQ:hdg_f_flux}\\
  = & \int_{\Edge} f v \dx +  \tau \int_{\partial \Edge} u_\textup D v \ds\notag
 \end{align}
\end{subequations}

Once $\lambda$ has been computed, the HDG approximation to \eqref{EQ:diffusion_mixed} on $\Graph$ will be computed as
\begin{equation}\label{EQ:construct_unk}
 U = \localU \lambda + \localU (f, u_\textup D), \qquad \vec Q = \localQ \lambda + \localQ (f, u_\textup D)
\end{equation}

The global coupling condition is derived through a discontinuous
Galerkin version of mass balance and reads: Find
$\lambda \in \skeletalSpace$, such that for all
$ \mu \in \skeletalSpace$
\begin{equation}
 \sum_{\edge \in \setEdge} \sum^{\node \in \setNode \setminus \setNodeDir}_{\node \in \edge} \int_\Node \left[ \vec Q \cdot \Normal + \tau (U - \lambda) \right] \mu \ds = \sum_{\node \in \setNode \setminus \setNodeDir} \int_\Node g \mu \ds.\label{EQ:hdg_global}
\end{equation}
\begin{remark}\label{REM:spaces}
Hybridized DG methods in dual mixed form differ by the choice of local polynomial spaces and the stabilization parameter $\tau$. Defining $\polynomials_p$ as the space of multivariate polynomials of degree at most $p$, Table~\ref{tab:hdg-methods} lists some well-known combinations on simplices.
\begin{table}[tp]
    \centering
    \begin{tabular}{l|ccc|c}
         Method &  $\hat M$ & $V_\Edge$ & $\vec W_\Edge$ & $\tau$ \\\hline
         LDG-H & $\polynomials_p$ & $\polynomials_p$ & $\polynomials^d_p$ & $>0$\\
         RT-H & $\polynomials_p$ & $\polynomials_p$
            & $\polynomials^d_p + \vec x \polynomials_p$ & $=0$\\
         BDM-H & $\polynomials_p$ & $\polynomials_{p-1}$
            & $\polynomials^d_p $ & $=0$\\
    \end{tabular}
    \caption{Combinations of local polynomial spaces and stabilization parameters for various hybridized methods.}
    \label{tab:hdg-methods}
\end{table}
Well-posedness of the local solvers for all of them is proven in \cite{CockburnGL2009} and the works cited there.
Analogous combinations based on tensor product polynomials exist for hypercubes.
\end{remark}
Existence and uniqueness of the discrete solution $\lambda$, $U$, and $\vec Q$ to the HDG method can be shown repeating the arguments mentioned in Section \ref{SEC:hdg_mixed} in the finite-dimensional setting. A natural assumption is the well-posedness of the local problems~\eqref{EQ:hdg_scheme}, see Remark~\ref{REM:spaces}.

Given the local solvers, the HDG method for elliptic diffusion problems is consistent with respect to the solution to~\eqref{EQ:diffusion_mixed}. Using consistency, we can immediately apply the analysis in \cite{CockburnGW2009}, as it proceeds locally for each hyperedge. Thus, we obtain optimal convergence rates for LDG-H (and also RT-H by slight adaptions) on simplicial hypergraphs. They also transfer to quadrilateral hypergraphs, since these allow for a Raviart--Thomas projection satisfying equation (2.7) in \cite{CockburnGW2009}.

\subsection{Numerical convergence tests for LDG-H}\label{SEC:diff_conv_ell}
\begin{table}[t]
 \begin{tabular}{cc|@{\,}lcc@{\,}lcc@{\,}lcc}
  \toprule
  \multicolumn{2}{c|@{\,}}{}      && \multicolumn{2}{c}{$\dimEdge=1$}  && \multicolumn{2}{c}{$\dimEdge=2$}   && \multicolumn{2}{c}{$\dimEdge=3$} \\
  \cmidrule{4-5} \cmidrule{7-8} \cmidrule{10-11}
  \multicolumn{2}{c|@{\,}}{mesh}  && err & eoc && err & eoc && err & eoc   \\
  \midrule
  \multirow{6}{*}{\rotatebox[origin=c]{90}{filling}}
  & $i = 0$ && 2.71e-1 & --- && 2.64e-1 & --- && 1.31e-1 & ---  \\
  & $i = 1$ && 9.82e-2 & 1.5 && 7.96e-2 & 1.7 && 3.24e-2 & 2.0  \\
  & $i = 2$ && 4.05e-2 & 1.3 && 2.55e-2 & 1.6 && 8.07e-3 & 2.0  \\
  & $i = 3$ && 1.81e-2 & 1.2 && 8.56e-3 & 1.6 && 2.01e-3 & 2.0  \\
  & $i = 4$ && 8.57e-3 & 1.1 && 2.94e-3 & 1.5 && 5.04e-4 & 2.0  \\
  & $i = 5$ && 4.16e-3 & 1.0 && 1.02e-3 & 1.5 && 1.26e-4 & 2.0  \\
  \midrule
  \multirow{4}{*}{\rotatebox[origin=c]{90}{refinement}}
  & $r = 0$ && 4.05e-2 & --- && 2.55e-2 & --- && 8.07e-3 & ---  \\
  & $r = 1$ && 1.00e-2 & 2.0 && 6.38e-3 & 2.0 && 2.01e-3 & 2.0  \\
  & $r = 2$ && 2.52e-3 & 2.0 && 1.59e-3 & 2.0 && 5.04e-4 & 2.0  \\
  & $r = 3$ && 6.30e-4 & 2.0 && 3.98e-4 & 2.0 && 1.26e-4 & 2.0  \\
  \bottomrule
 \end{tabular}\vspace{1ex}
 \caption{$L^2$ errors (err) and estimated orders of convergence (eoc) of linear approximation to the diffusion equation for hypergraphs with hyperedge dimension $\dimEdge$.}\label{TAB:diff_hg_el_conv}
\end{table}
\begin{figure}
 \begin{tabular}{ccc}
  \includegraphics[width=0.3\textwidth,draft=false]{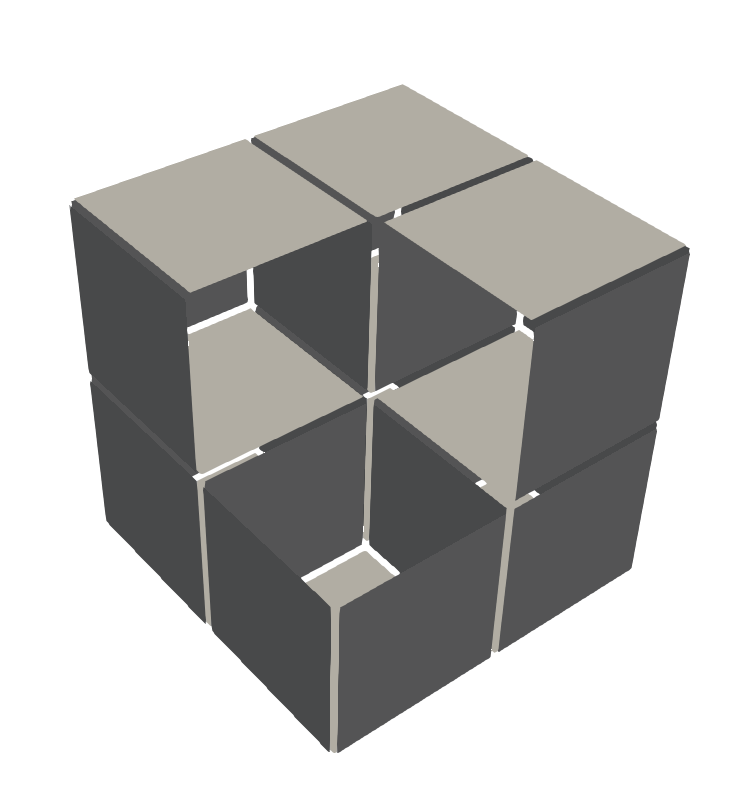} & \includegraphics[width=0.3\textwidth,draft=false]{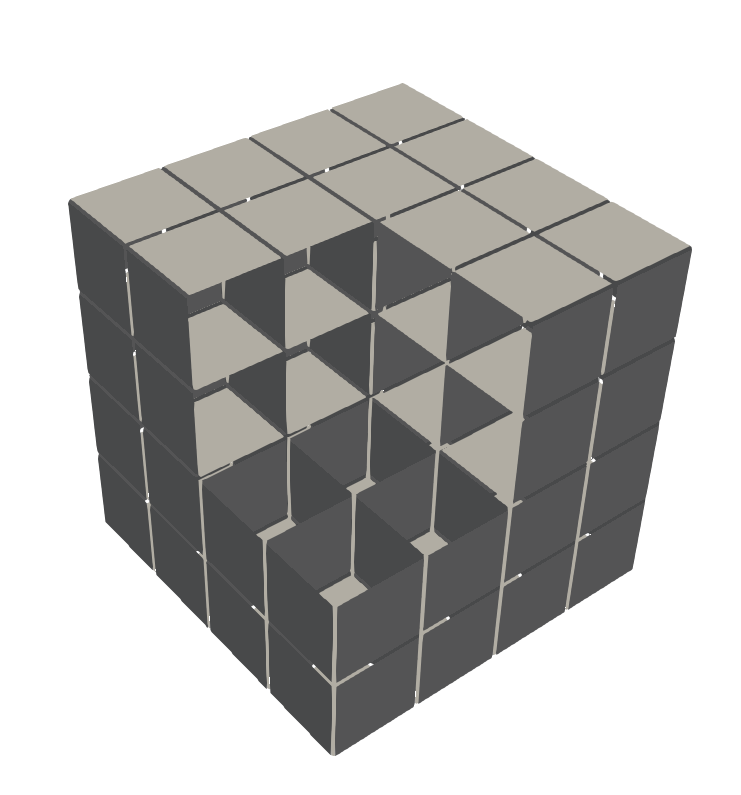} & \includegraphics[width=0.3\textwidth,draft=false]{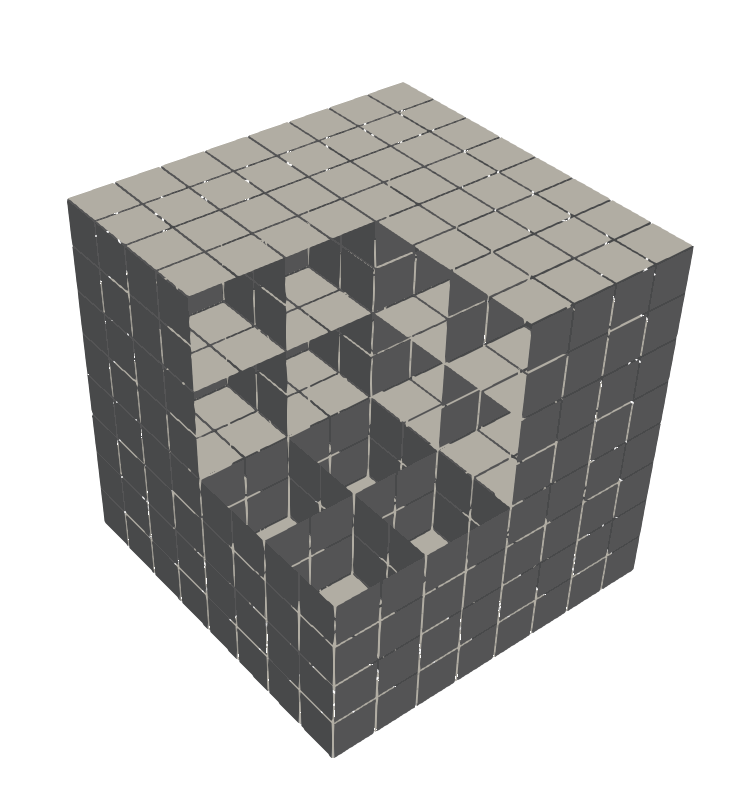} \\
  $i = 1$ & $i = 2$ & $i = 2$ \\
  $r = 0$ & $r = 0$ & $r = 1$ 
 \end{tabular}
 \caption{Pictures of the computational domains for $\locDim = 2$, and different combinations of filling $i$ and refinement $r$. I all three illustrations, the same corner of the ``cubes" has been removed in the plots to illustrate the interior structure of the ``cubes".}\label{FIG:diff_hg_el_conv}
\end{figure}
Next, we consider a convergence example on a hypergraph. It is constructed to approximate
\begin{equation}\label{EQ:hypergraph_example}
 - \nabla \cdot (\kappa \nabla u) = f \text{ in } \Edge \in \SetEdge, \qquad u = u_\text D \text{ on } \Node \in \SetNodeDir,
\end{equation}
where the Dirichlet nodes are those that are located on the boundary of $[0,1]^\globDim$ with $\globDim = 3$.

The filling $i$ indicates that the cube has been $i$ times uniformly refined (in the standard three dimensional sense), and the calculation is conducted on the $\locDim$ dimensional ``surfaces" of this filling. These surfaces themselves might be further refined $r$ times, and these refined surfaces are identified to be our standard hyperedges, see Figure \ref{FIG:diff_hg_el_conv} for an illustration.

The $\locDim-1$ dimensional faces of this approach are interpreted as nodes and the nodes located on the boundary of the unit cube are considered Dirichlet nodes. All other nodes are supposed to be interior nodes. The solution is constructed to be $u = -x^2 - y^2 - z^2$, diffusion coefficient $d = 1$, and right-hand side $f = 2 \locDim$. Of course, polynomial degrees $\ge 2$ are supposed to exactly reproduce the given solution, which is true in our numerical experiments. Thus, we only plot the errors for $p = 1$ in Table \ref{TAB:diff_hg_el_conv}.

Interestingly, the $L^2$ errors converge although with filling $i$, also the computational domain increases for $\locDim = 1$ and $\locDim = 2$. However, the rate of convergence deteriorates by $1$ if $\locDim = 1$, and $\tfrac{1}{2}$ if $\locDim = 2$. The optimal order is obtained for $\locDim = 3$. 

Beyond this, the refinement indicated with $r$ uses filling level $i = 2$ and then uniformly refines the respective faces. This does not lead to an increase of the computational domain (even if $\locDim < \globDim$) and, therefore, gives the optimal convergence rate $p = 2$.

The aforementioned results have been obtained using our code HyperHDG \cite{HyperHDGgithub}.
%
\section{Hypergraph PDE as singular limit}
\label{SEC:limit}
\input{limit}
%
\section{Conclusions}\label{SEC:conclusions}
%
We have motivated the formulation of PDEs on geometric hypergraphs, which generalize the notions of ``domains'', ``graphs'', and ``network of surfaces''. Using a simple singular limit example in which several thin subdomains meet in a common inter-domain, we underlined that problems on hypergraphs might evolve from practical applications, and that the hybrid formulations of PDEs is particular useful for their expression. Thus, we used hybrid methods (in particular HDG methods), which intrinsically fit to this formulation, to approximate the solution of simple PDEs on hypergraphs. Doing so, we obtained the optimal convergence behavior, which is predicted by the theory of these methods applied to PDEs on standard domains.
%
\section*{Acknowledgements}
%
This work is supported by the German Research Foundation under Germany's Excellence Strategy EXC 2181/1 - 390900948 (the Heidelberg STRUCTURES Excellence Cluster).

The authors thank Dominic Kempf and the Scientific Software Center of Heidelberg University for their support in developing HyperHDG.
\bibliographystyle{ARalpha}
\bibliography{HyperHDG}
\end{document}

%% file: limit.tex
\begin{figure}
 \includegraphics[width=.8\textwidth]{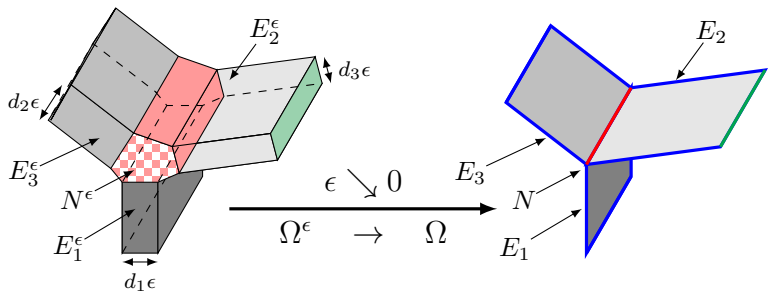}
 \caption{Model problem of a hypergraph as singular limit.
 Extruded domain with \textcolor{red}{red} $\Node$ ($\omega$ is \textcolor{red}{red} checkerboard), and homogeneous Dirichlet boundary \textcolor{Green}{green} on the left.
 Singular limit with $\Node$ depicted \textcolor{red}{red}, and homogeneous Dirichlet node highlighted \textcolor{Green}{green} on the right.}\label{fig:hg_limit}
\end{figure}
The aim of this section is to derive the hypergraph model \eqref{EQ:diffusion_primal} as a singular limit of a 3D-model problem as illustrated in Figure \ref{fig:hg_limit}.
We exemplary use the figure to illuminate the basic ideas: We assume to have a diffusion problem on a domain consisting of three thin plates (in gray) and a (red) joint. This is the problem, which we would like to solve. However, we do not want to solve it in three spatial dimensions, but would like to reduce it to a two-dimensional problem---for example since we have limited compute sources, the domain is very large, or the domain is very complicated to mesh. Thus, we let the thickness of the three plates (and therefore also the thickness of the joint) go to zero by considering $\epsilon \searrow 0$, and construct a two dimensional limit problem. The solution of this two dimensional problem lives on the mid-planes of the three planes and their joint. It in some sense is supposed to approximate the solution of the original (three-dimensional) problem for which $\epsilon$ is a small, positive number.

The principal idea of the limit process is to map equations on the thin structures depending on $\epsilon$ to fixed reference domains independent of $\epsilon$, where we can use standard compactness methods from functional analysis. However, the transformed problem includes $\epsilon$-dependent coefficients. Thus, the crucial point for the derivation of the limit model is to establish \textit{a priori} estimates that are uniform with respect to $\epsilon > 0$.
\subsection{Description of the 3D model problem}
We consider the simplified case of one hypernode $\Node$ with length $L$ connecting $m$ hyperedges $\Edge_i$ for $i=1,\ldots,m$ which are rectangles with side lengths $L$ and $L_i$. Thus, Figure \ref{fig:hg_limit} shows the case with $m=3$. The opposite node of $\Node$ with respect to $\Edge_i$ is denoted by $\Node_{i,e}$ (and is a boundary node).  Without loss of generality, we assume that $\Node$ lies in the $x_1$-axis and we have 
\begin{gather}
    \Node= \{s e_1 \, : \, s \in (0,L)\}.
\end{gather}
We denote by $\nu_i$ a unit normal vector to $E_i$ and define extruded hyperedges for $0 <\epsilon \ll 1$ and $0 < d_i$
\begin{equation*}
\tilde{E}_i^{\epsilon}:= \left\{ x \in \IR^3 \,: \, x = y + s\nu_i \mbox{ for } s \in \left(-\dfrac{d_i \epsilon}{2},\dfrac{d_i \epsilon}{2}\right), \, y \in E_i \right\}.
\end{equation*}
Hence, $\tilde{E}_i^{\epsilon}$ is a hexahedron with side lengths $L$ and $L_i$, and with thickness $d_i \epsilon$. We construct now a domain $\Omega^\epsilon$ which contains the union of all these extruded hyperedges and a nonoverlapping decomposition of this domain.
To this end, let $\alpha>0$ be chosen such that the sets 
\begin{gather}
  \label{EQ:def-eei}
  E_i^{\epsilon}:= \left\{x \in \tilde{E}_i^{\epsilon} \, : \, \mathrm{dist}(x,\tilde{S}_i^{\epsilon}) > \alpha \epsilon \right\},
\end{gather}
do not overlap.
We denote the side of $\tilde{E}_i^{\epsilon}$ that contains $\Node$ by $\tilde{S}_i^{\epsilon}$, and define
\begin{align*}
S_i^{\epsilon}&:= \mathrm{int}\left\{x \in \partial E_i^{\epsilon} \, : \,  \mathrm{dist}(x,\tilde{S}_i^{\epsilon}) = \alpha \epsilon \right\}.
\end{align*}
The side lengths of $E_i^{\epsilon}$ are $L$ and $L_i - \alpha\epsilon$.
Additionally, we define the convex hull of the node $\Node$ and the sides $S_i^{\epsilon}$:
\begin{equation*}
N^{\epsilon}:= \mathrm{int}\left( \mathrm{conv}\{\overline{\Node},\overline{S_1^{\epsilon}},\ldots,\overline{S_m^{\epsilon}}\}\right).
\end{equation*}
By construction, we have
\begin{equation*}
N^{\epsilon} = (0,L)\times \epsilon \omega.
\end{equation*}
Then, we define the thin domain $\Omega^{\epsilon}$ as
\begin{equation*}
\Omega^{\epsilon} := N^{\epsilon} \cup \bigcup_{i=1}^m \left( E_i^{\epsilon} \cup S_i^{\epsilon}\right).
\end{equation*}
On $\Omega^{\epsilon}$ we define a  diffusion problem and then to pass to the limit $\epsilon \searrow 0$ in order to derive a problem on the hypergraph $\Omega = N \cup \bigcup_{i=1}^m E_i$.
To ensure uniqueness for our model we assume a zero-Dirichlet boundary condition on one face $S_D^{\epsilon}$.  We consider the following problem for the unknown $\ueps: \Omega^{\epsilon} \rightarrow \IR$
\begin{subequations}\label{EQ:diff_eps_model}
\begin{align}
- \nabla \cdot \big( \kappa^{\epsilon} \nabla \ueps \big) &= f^{\epsilon} &\mbox{ in }& \Omega^{\epsilon},
\\
\ueps &= 0 &\mbox{ on }& S_D^{\epsilon},
\\
-\kappa^{\epsilon} \nabla \ueps \cdot \nu &= 0 &\text{ on }& \partial \Omega^{\epsilon} \setminus  S_D^{\epsilon}.
\end{align}
\end{subequations}
Here we assume that $\kappa^{\epsilon}$ is piecewise constant on $\Edge^{\epsilon}_i$ and $\Node$, i.e. we have $\kappa^{\epsilon}\vert_{\Edge^{\epsilon}_i} = \kappa_i  > 0$ and $\kappa^{\epsilon}\vert_{\Node} = \kappa_{\Node}  > 0$.  For $f^{\epsilon}$ we assume that $f^{\epsilon}\vert_{\Edge^{\epsilon}_i} = f_i$ with $f_i \in C^0(\IR^3)$ and 
  \begin{gather}
      f^\epsilon|_{\Node^\epsilon} (x)=
      \frac1{\epsilon\lvert\omega\rvert} g(x)
  \end{gather}
with $g \in C^0(\IR^3)$.  
\begin{remark}\label{REM:scale_f}
 We emphasize that the function $f^{\epsilon}_N$ is of order $\epsilon^{-1}$. That is, a nonzero source term $g$ on a hypernode can only be caused by a large sink/source which converges to a measure on the hypernode.
\end{remark}
\begin{definition}
Let
\begin{gather}
    H^1_D(\Omega^{\epsilon}) =
    \bigl\{ u\in H^1(\Omega^{\epsilon}) \bigm\lvert \ueps_{|S_D^{\epsilon}} = 0
    \bigr\}.
\end{gather}
We call $\ueps \in H^1_D(\Omega^{\epsilon})$ a weak solution of $\eqref{EQ:diff_eps_model}$  if for all $\phi^{\epsilon} \in H^1_D(\Omega^{\epsilon})$ there holds
\begin{align}\label{EQ:Var_eps}
\int_{\Omega^{\epsilon}} \kappa^{\epsilon} \nabla \ueps \cdot \nabla \phi^{\epsilon} \dx
= \int_{\Omega^{\epsilon}} f^{\epsilon} \phi^{\epsilon} \dx.
\end{align}
\end{definition}
As the standard solution theory for elliptic equations applies to $\Omega^\epsilon$, the Lax--Milgram Lemma immediately implies the existence of a unique weak solution of $\eqref{EQ:diff_eps_model}$ for all $\epsilon > 0$.
\subsection{Transformation to $\epsilon$ independent domains}
First, we map the domains $E_i^{\epsilon}$ and $N^{\epsilon}$ to fixed domains $\Eref$ and $\Nref$, respectively:
\begin{align*}
 \Eref &:= (0,L) \times (0,L) \times \left(-\frac12,\frac12\right),\\
 \Nref &:= (0,L) \times \omega.
\end{align*}
For $N^{\epsilon}$ we just use a simple scaling:
\begin{equation*}
\Phi_N^{\epsilon}: N^{\epsilon} \rightarrow \Nref,\qquad \Phi_N^{\epsilon}(x)= \ANeps x,
\end{equation*}
where the matrix $\ANeps \in \IR^{3 \times 3} $ is given via
\begin{align*}
\ANeps = \begin{pmatrix}
1 & 0 & 0 \\
0  & \epsilon^{-1} & 0 \\
0 & 0 & \epsilon^{-1}
\end{pmatrix}.
\end{align*}
The transformation between $E_i^{\epsilon}$ and $\Eref$ is defined by  ($i=1,\ldots,m$)
\begin{align}\label{TransformationReferenceElement}
\Phi_i^{\epsilon} : E_i^{\epsilon} \rightarrow \Eref, \qquad \Phi_i^{\epsilon}(x) = \Aieps R_i x + a_i^{\epsilon},
\end{align}
where $R_i \in \IR^{3\times 3}$ is a rotation around the $x_1$-axis, $\Aieps \in \IR^{3\times 3}$ and $a_i^{\epsilon} \in \IR^3$ are given by
\begin{align*}
R_i = \begin{pmatrix}
1 & 0 \\
0 & \bar{R}_i 
\end{pmatrix},
\quad
\Aieps = \begin{pmatrix}
1 & 0 & 0 \\
0 & \frac{L}{L_i - \alpha\epsilon} & 0 \\
0 & 0 & \frac{1}{d_i \epsilon} 
\end{pmatrix}, \quad a_i^{\epsilon} = \begin{pmatrix}
0 \\ - \frac{L}{L_i - \alpha\epsilon} \alpha\epsilon \\ 0
\end{pmatrix},
\end{align*}
where the matrix $\bar{R}_i \in \IR^{2\times 2}$ is a rotation.
More precisely, $R_i$ can be determined by the equation $R_i\nu_i = e_3$, where $e_3$ is the unit vector in $x_3$-direction and $\nu_i$ is a unit normal on $E_i$.

Further, we write 
\begin{align*}
\Psi_N^{\epsilon}(x) &:= (\Phi_N^{\epsilon})^{-1}(x) = (\ANeps)^{-1} x, 
\\
\Psi_i^{\epsilon} (x) &:= (\Phi_i^{\epsilon})^{-1}(x) = R_i^{-1} (\Aieps)^{-1} \big(x - a_i^{\epsilon}\big).
\end{align*} 
The side of $\Eref$ with the homogeneous Dirichlet boundary condition on $\Edge_m$ is denoted by $S_D$.
It is characterized via
\begin{align*}
S_D := \Phi_i^{\epsilon}(S_{D}^{\epsilon})= (0,L)\times \{L\} \times \left(-\frac12,\frac12\right).
\end{align*} 
The sides of $\Nref$ where $\Node^\epsilon$ interfaces to the hyperedges are
\begin{align*}
S_i = \Phi_N^{\epsilon}(S_i^{\epsilon}) \subset (0,L)\times \partial \omega.
\end{align*}

Now we make a change of variables to transform equation $\eqref{EQ:Var_eps}$ to the fixed domains $\Eref$ and $\Nref$. Then we define
\begin{align*}
\ueps_i(x) &:= \ueps \big(\Psi_i^{\epsilon}(x)\big) &&\mbox{ for almost every } x \in \Eref,
\\
\ueps_N(x) &:= \ueps \big(\Psi_N^{\epsilon} (x)\big) &&\mbox{ for almost every } x \in \Nref.
\end{align*}
In other words, we identify the function $\ueps \in H^1_D(\Omega^{\epsilon})$ with the tuple 
\begin{align*}
(\ueps_1,\ldots,\ueps_m, \ueps_N) \in H^1(\Eref)^m \times H^1(\Nref),
\end{align*}
together with the interface and boundary conditions
\begin{align}\label{EQ:inter_bound_eps}
\ueps_N = \ueps_i \circ \Phi_i^{\epsilon} \circ \Psi_N^{\epsilon} \,\, \mbox{ on } S_i, \qquad 
\ueps_m = 0 \,\, \text{ on } S_D.
\end{align}

Next, we use the fact that
\begin{subequations}
\begin{equation}\label{EQ:def_s}
 (\Phi_i^{\epsilon} \circ \Psi_N^{\epsilon})|_{S_i}\colon S_i \to S := (0,L)\times \{0\} \times \left(-\frac12,\frac12\right)
\end{equation}
is an isomorphism between a face of $\Nref$ and $\Eref$. Thus, we can write it in the $\epsilon$-independent form
\begin{equation}\label{EQ:Composition}
 (\Phi_i^{\epsilon} \circ \Psi_N^{\epsilon})|_{S_i} = \underbrace{ \begin{pmatrix} 1  & 0 & 0 \\ 0 & 1 & 0 \\ 0 & 0 & d_i^{-1} \end{pmatrix} }_{ =: A_i } R_i x - de_2.
\end{equation}
\end{subequations}
Note that here, the first and second component of the mapping can be obtained by a simple concatenation of $\Phi^\epsilon_i$ and $\Psi^\epsilon_N$, in particular, the second component is zero ---cf.\ the definition of S in \eqref{EQ:def_s}, since $(R_i x)_2 = d$ for all $x \in S_i$.

By a change of coordinates and an elemental calculation we obtain
\begin{align}
\epsilon^2 \int_{\Nref} & \kappa_{\Node} (\ANeps)^2 \nabla \ueps_N \cdot \nabla \phi_N \dx \notag\\
& + \sum_{i=1}^m \frac{(L_i - \alpha\epsilon) d_i \epsilon}{L}\int_{\Eref} \kappa_i (\Aieps)^2 \nabla \ueps_i \cdot \nabla \phi_i \dx 
\label{EQ:var_eps_ref}\\
= &\frac{\epsilon}{\vert \omega\vert} \int_{\Nref}   g \circ \Psi^{\epsilon}_N \phi_N \dx 
+ \sum_{i=1}^m \frac{(L_i - \alpha\epsilon) d_i \epsilon}{L} \int_{\Eref} f_i \circ \Psi^{\epsilon}_i\phi_i \dx\notag
\end{align}
for all $\phi =(\phi_1,\ldots,\phi_m,\phi_N) \in H^1(\Eref)^m \times H^1(\Nref)$ which fulfills the boundary and interface conditions in $\eqref{EQ:inter_bound_eps}$.
\subsection{A priori estimates that are uniform in $\epsilon$}
In the following we use for $x \in \IR^3$ the notation $x_\parallel = (x_1,x_2)$ and $x_{\perp} = (x_2,x_3)$, as well as $\nabla_{\parallel} = (\partial_1,\partial_2)^T$ and $\nabla_{\perp} = (\partial_2,\partial_3)^T$.

In a first step, we derive \textit{a priori} estimates for $\ueps_N$ and $\ueps_i$ uniformly with respect to $\epsilon$. We define the space
\begin{align*}
H^1(\Nref,\nabla_{\perp}) := \left\{\phi_N \in L^2(\Nref) \, : \, \nabla_{\perp} \phi_N \in L^2(\Nref)^2 \right\},
\end{align*}
with the norm $\Vert \phi_N \Vert^2_{H^1(\Nref,\nabla_{\perp})} := \Vert \phi_N\Vert_{L^2(\Nref)}^2 + \Vert \nabla_{\perp} \phi_N\Vert_{L^2(\Nref)}^2$. The proof of the trace theorem in \cite[Chapter 5.5, Theorem 1]{EvansPartialDifferentialEquations} implies the existence of a bounded linear trace operator
\begin{align*}
    \gamma_{\nabla{\perp}} : H^1(\Nref,\nabla_{\perp}) \rightarrow L^2(S_N)
\end{align*}
for $S_N:= (0,L)\times \partial \omega$.
We use the abbreviations $\phi_{N\vert S_N} := \gamma_{\nabla_{\perp}}(\phi_N)$ and $\phi_{N\vert S_i}$ for the restriction to $S_i$, respectively. Additionally, we have the following Poincaré-inequality:
\begin{lemma}\label{LemmaPoincareInequality}
For all $v = (v_1,\ldots,v_m,v_N) \subset H^1(\Eref)^m \times H^1(\Nref)$ with the boundary conditions $\eqref{EQ:inter_bound_eps}$ it holds 
\begin{multline*}
\sum_{i=1}^m\Vert v_i \Vert_{L^2(\Eref)}+\Vert v_N \Vert_{L^2(\Nref)} \\
\le C \left(\Vert \nabla_{\perp} v_N \Vert_{L^2(\Nref)} + \sum_{i=1}^m \Vert \nabla v_i \Vert_{L^2(\Eref)} \right).
\end{multline*}
\end{lemma}
\begin{proof}
As for the  Poincaré-inequality in the proof of Theorem \ref{TH:ex_uni} we use a contradiction argument. We assume that there exists a sequence  $(v_1^n,\ldots,v_m^n,v_N^n)_n \subset H^1(\Eref)^m \times H^1(\Nref)$ with the boundary conditions $\eqref{EQ:inter_bound_eps}$, such that
\begin{align}
\begin{aligned}\label{EQ:Poincare_Contradiction}
1 = \sum_{i=1}^m\Vert v_i^n \Vert_{L^2(\Eref)}+\Vert v_N^n & \Vert_{L^2(\Nref)}  
\\
&\geq n \left(\Vert \nabla_{\perp} v_N^n \Vert_{L^2(\Nref)} + \sum_{i=1}^m \Vert \nabla v_i^n \Vert_{L^2(\Eref)} \right).
\end{aligned}
\end{align}
Since $v_i^n$ is bounded in $H^1(\Eref)$ and $v_N^n$ is bounded in $H^1(\Nref,\nabla_{\perp})$, there exist $v_i$ and $v_N$, such that up to a subsequence 
\begin{align*}
v_i^n &\rightharpoonup v_i &\mbox{ weakly in }& H^1(\Eref),
\\
v_i^n &\rightarrow v_i &\mbox{ in }& L^2(\Eref),
\\
v_N^n &\rightharpoonup v_N &\mbox{ weakly in }& H^1(\Nref,\nabla_{\perp}).
\end{align*}
Further, due to \eqref{EQ:Poincare_Contradiction} we have $\nabla v_i = 0 $ and $\nabla_{\perp} v_N = 0$. This implies that $v_i$ is constant on $\Eref$ (and $v_m = 0$ due to the zero boundary condition on $S_D$) and there exists $\bar{v}_N \in L^2(N)$ such that $v_N(x) = \bar{v}_N(x_1)$ for almost every $x \in \Nref$. The continuity of the usual trace operator on $H^1(\Eref)$ and the continuity of $\gamma_{\nabla_{\perp}}$ imply that $v_i(A_iR_ix - de_2)\vert_{S_i} = \bar{v}(x_1)$ on $S_i$.
Further, the continuity of the trace operator implies the weak convergence of the traces (from both sides). Additionally we used, that we have $v_N^n\vert_{S_i} = v_i^n\vert_{S_i}$, which can be shown by a density argument. This implies $v_i = 0$ and $v_N = 0$, which contradicts  $1 = \sum_{i=1}^m\Vert v_i \Vert_{L^2(\Eref)}+\Vert v_N  \Vert_{L^2(\Nref)}  $.
\end{proof}
Now, we obtain the following \textit{a priori} estimates:
\begin{lemma}\label{LE:apriori_estimates}
For $\ueps_i$ and $\ueps_N$ it holds that
\begin{align*}
\epsilon & \Vert \partial_1 \ueps_N \Vert_{L^2(\Nref)} + \Vert \nabla_{\perp} \ueps_N \Vert_{L^2(\Nref)} 
\\
&+ \sum_{i=1}^m \left\{ \sqrt{\epsilon} \Vert \nabla_{\parallel} \ueps_i \Vert_{L^2(\Eref)} + \Vert \partial_3 \ueps_i \Vert_{L^2(\Eref)} \right\} \le C\sqrt{\epsilon}
\end{align*}
for a constant $C > 0$ independent of $\epsilon$.
\end{lemma}
\begin{proof}
Choosing  $\phi_i = \ueps_i$ and $\phi_N = \ueps_N$ as a test-function in $\eqref{EQ:var_eps_ref}$ and using the positivity of $\kappa^{\epsilon}$ and the assumptions on $f^{\epsilon}$, as well as the trace-inequality and the Poincar\'e-inequality from Lemma \ref{LemmaPoincareInequality}, we obtain
\begin{align*}
\epsilon^2 &\Vert \partial_1 u_N^{\epsilon}\Vert_{L^2(\Nref)}^2 + \Vert \nabla_{\perp} \ueps_N \Vert_{L^2(\Nref)}^2
\\
&+ \sum_{i=1}^m \left\{ \epsilon \Vert \nabla_{\parallel} \ueps_i \Vert_{L^2(\Eref)}^2 + \Vert \partial_3 \ueps_i \Vert_{L^2(\Eref)} \right\} 
\\
\le & C \epsilon \Vert g \circ \Psi_{\Node}^{\epsilon} \Vert_{L^2(\Nref)} \Vert \ueps_N \Vert_{L^2(\Nref)} + C\epsilon \sum_{i=1}^m \Vert f_i \circ \Psi_i^{\epsilon} \Vert_{L^2(\Eref)} \Vert \ueps_i \Vert_{L^2(\Eref)} 
\\
\le & C \epsilon \Vert \ueps_N \Vert_{L^2(\Nref)} + C\epsilon \sum_{i=1}^m \Vert \ueps_i \Vert_{L^2(\Eref)} 
\\
\le & C\epsilon + \frac12 \left(  \Vert \nabla_{\perp} \ueps_N \Vert_{L^2(\Nref)}^2 + \epsilon \sum_{i=1}^m \Vert \nabla  \ueps_i \Vert_{L^2(\Eref)}^2 \right)
\end{align*}
The second term on the right-hand side can be absorbed from the left-hand side and we obtain the desired result.
\end{proof}
\subsection{Convergence by compactness results and characterization of limit problem}
Using the weak compactness of the unit ball in $L^2$ and the Rellich-Kondrachov theorem we immediately obtain the following compactness result:

\begin{corollary}\label{CO:convergences_eps}
There exists $u_N^0 \in L^2(\Nref)$ with $\nabla_{\perp} u_N^0 = 0$ in the weak sense, and $u_i^0 \in H^1(\Eref)$ for $i=1,\ldots,m$, such that up to a subsequence
\begin{align*}
\ueps_N &\rightharpoonup u_N^0 &\mbox{ weakly in }& L^2(\Nref),
\\
\ueps_i &\rightarrow u_i^0 &\mbox{ strongly in }& L^2(\Eref),
\\
\nabla \ueps_i &\rightharpoonup \nabla u_i^0 &\mbox{ weakly in }& L^2(\Eref),
\\
\nabla_{\perp} \ueps_N &\rightarrow 0 &\mbox{ strongly in }& L^2(\Nref),
\\
\partial_3 \ueps_i &\rightarrow 0 &\mbox{ strongly in }& L^2(\Eref).
\end{align*}
Especially the last two convergences imply $\nabla_{\perp} \tilde{u}_N^0 = 0$ and $\partial_3 \tilde{u}_i^0 = 0$ in the weak sense. Hence, there exist $u_N^0 \in L^2(N)$ and $u_i^0 \in H^1((0,L)^2)$ such that $\tilde{u}_N^0(x) = u_N^0(x_1)$  for almost every $x \in \Nref$ and $\tilde{u}_i^0(x) = u_i^0(x_\parallel)$ for almost every $x \in \Eref$.
\end{corollary}
In the following, we drop the notation $\tilde{\cdot}$ and just use the notation $u_N^0 \in L^2(\Nref)$ and $u_i^0\in H^1(\Eref)$ for the limit functions. Let us consider the interface and boundary conditions for these limits. Obviously, the zero boundary condition $\ueps_m = 0 $ on $S_D$ is inherited to  $u_0^m$ by the continuity of the trace operator on $H^1(\Eref)$. 

The weak convergence of $\ueps_N$ in $H^1(\Nref,\nabla_{\perp})$ to $u_N^0$ implies the weak convergence $\ueps_N\vert{S_i} \rightharpoonup u_N^0 $ in $L^2(S_i)$. Further, the weak convergence of $\ueps_i $ in $H^1(\Eref)$ and the compactness of the embedding $H^1(\Eref) \hookrightarrow L^2(\partial \Eref)$ implies the strong convergence   $\ueps_i\vert_{S_i} \rightarrow u_i^0 $ in $L^2(S_i)$. Using $\eqref{EQ:inter_bound_eps}$ and $\eqref{EQ:Composition}$ we obtain for all $\phi \in C_0^{\infty}(S_i)$
\begin{align*}
\int_{S_i} u_N^0 \phi d\sigma &= \lim_{\epsilon\to 0} \int_{S_i} \ueps_N \phi d\sigma 
= \lim_{\epsilon\to 0} \int_{S_i} \ueps_i (A_iR_i x - d e_2) \phi d\sigma
\\
&= \int_{S_i} u_i^0(A_iR_i x - de_2) \phi d\sigma.
\end{align*}
This implies the interface condition
\begin{align}\label{EQ:inter_limit}
u_0^N\vert_{S_i}(x) = u_i^0\vert_{S}  (A_i R_i x - de_2) \quad \mbox{ for almost every } x \in S_i.
\end{align}
Now, let us pass to the limit in the variational equation $\eqref{EQ:var_eps_ref}$ for suitable test-functions. We choose $\phi_N \in C^{\infty}(\overline{N})$ (hence $\phi_N $ is constant on every section $\{x_1\}\times \omega$ of $\Nref$) and $\phi_i \in C^{\infty}(\overline{(0,L)^2})$ with $\phi_i = \phi_N$ on $ (0,L)\times \{0\} $ ($\phi_i$ is constant in $x_3$-direction) and $\phi_m $ has compact support away from $S_D$. Obviously, $(\phi_1,\ldots,\phi_m,\phi_N)$ is an admissible test-function for $\eqref{EQ:var_eps_ref}$. We multiply $\eqref{EQ:var_eps_ref}$ with $1/\epsilon$ and obtain
\begin{align}
\begin{aligned}
\label{EQ:aux_var_eps_limit}
\epsilon &\int_{\Nref} \kappa_N \partial_1 \ueps_N \partial_1 \phi_N \dx 
+ \sum_{i=1}^m \frac{(L_i - \alpha\epsilon) d_i }{L}\int_{\Eref} \kappa_i (\Aieps)^2 \nabla \ueps_i \cdot \nabla \phi_i \dx 
\\
=&  \int_{\Node}  g \circ \Psi^{\epsilon}_N \phi_N \dx_1
+ \sum_{i=1}^m \frac{(L_i - \alpha\epsilon) d_i }{L} \int_{\Eref} f_i \circ \Psi^{\epsilon}_i \phi_i \dx.
\end{aligned}
\end{align}
Now, we pass to the limit $\epsilon \to 0$ in every single term. From Lemma \ref{LE:apriori_estimates}  we immediately obtain that the first term on the left-hand side is of order $\sqrt{\epsilon}$ and vanishes for $\epsilon \to 0$. 
The  convergence of $\nabla \ueps_i$ from Corollary \ref{CO:convergences_eps} implies
\begin{align*}
 \frac{(L_i - \alpha\epsilon) d_i }{L}&\int_{\Eref} \kappa_i (\Aieps)^2 \nabla \ueps_i \cdot \nabla \phi_i \dx 
\\
\overset{\epsilon \to 0}{\longrightarrow}& \frac{L_i d_i}{L} \int_{(0,L)^2} \kappa_i \left( \partial_1 u_i^0 \partial_1 \phi_i + \frac{L^2}{L_i^2}\partial_2 u_i^0 \partial_2 \phi_i \right) \dx_\parallel.
\end{align*}
Further, we have 
\begin{align*}
  \frac{(L_i - \alpha\epsilon) d_i }{L} \int_{\Eref} f_i \circ \Psi^{\epsilon}_i \phi_i \dx \overset{\epsilon \to 0 }{\longrightarrow} \frac{L_i}{L}d_i \int_{(0,L)^2} f_i \circ \Psi_i^0 \phi_i \dx_\parallel,  
\end{align*}
with 
\begin{align*}
    \Psi_i^0(x):= R_i^{-1} \begin{pmatrix}
1 & 0 & 0 \\
0 & \frac{L}{L_i} & 0 \\
0 & 0 & 0
\end{pmatrix}x,
\end{align*}
and in a similar way we get
\begin{align*}
    \int_{\Node}  g \circ \Psi^{\epsilon}_N \phi_N \dx_1 \overset{\epsilon \to 0 }{\longrightarrow} \int_{\Node} g(x_1,0,0) \phi_N \dx_1.
\end{align*}
Now, we define the function $u^0$ on the hypergraph $\Omega$ in the following way: For almost every $x \in E_i $ we define (we emphasize that $(R_ix)_3 = 0$)
\begin{align*}
u^0(x) := u_i^0 (C_iR_i x) \quad \mbox{ with } C_i = \begin{pmatrix}
1 & 0 & 0 \\
0 & \frac{L}{L_i} & 0 \\
0 & 0 & 0
\end{pmatrix}.
\end{align*}
Hence, we have $u^0 \in \spaceH$ with $u^0\vert E_m = 0$ on $\Node_{m,e}$ (see Definition \ref{DEF:weak_sol}), and altogether, we obtain for $\epsilon \to 0$ in $\eqref{EQ:aux_var_eps_limit}$ after a change of coordinates for all $\phi \in \spaceH$ with $\phi_m = 0$ on $\Node_{m,e}$
\begin{multline}
\sum_{i=1}^m d_i \int_{E_i} \kappa_i \nabla_{E_i} u^0_i \cdot \nabla_{E_i} \phi \dx
= \\
 \int_{\Node} g(x_1,0,0)\phi_{\vert\Node} \dx_1
+  \sum_{i=1}^n d_i \int_{\Edge_i}f_i \phi_i \dx.
\end{multline}
In other words, $u^0$ is the unique weak solution of ($\tilde\kappa_i=\tilde\kappa_{\vert E_i}= d_i \kappa_i$)
\begin{subequations}
\begin{align}
- \nabla_{E_i} \cdot (\tilde\kappa_i \nabla_{E_i} u^0) &= d_if_i &\mbox{ for }& i=1,\ldots,m,
\\
u^0 &= 0 &\mbox{ on }& \Node_{m,e},
\\
-  \jump{\tilde\kappa  \nabla_{E} u^0 \cdot \mathbf{n}} &=g &\mbox{ on }& N,
\\
-\tilde\kappa_i \nabla_{E_i} u^0 \cdot \mathbf{n} &= 0 &\mbox{ on }& \Node_{i,e} \mbox{ for } i=1,\ldots,m-1,
\\
u^0\vert_{E_i} &= u^0\vert_{E_j} &\mbox{ on }& \Node \mbox{ for } i,j=1,\ldots,m.
\end{align}
\end{subequations}
\subsection{Concluding remarks}
For the ease of presentation, we have just performed the limit analysis for a very simple model case. Therefore, we have to discuss the validity of our analysis in more general cases.
\begin{enumerate}[leftmargin=*]
    \item Different combinations of boundary conditions, Dirichlet, (inhomogeneous) Neumann, or Robin will yield the same result as long as the variation of boundary functions in direction normal to the hyperedges $\Edge_i$ vanishes for $\epsilon \searrow 0$.
    \item A hyperedge with several ``interior'' hypernodes: after changing the definition of $\Edge_i^\epsilon$ in~\eqref{EQ:def-eei} and the reference mapping $\Phi_i^\epsilon$ in~\eqref{TransformationReferenceElement}, two hypernode domains $\Node_i^\epsilon$ meeting in the same corner will have an intersection, which we have to treat separately in order to have a nonoverlapping decomposition of $\Omega^\epsilon$. But then, the integrals over these corner domains converge to zero by one order faster than the first term in~\eqref{EQ:aux_var_eps_limit}. Thus, they do not enter the limit equation. Obviously, not all nodes can be located in the $x$-axis, but an affine transformation to the reference node $\Nref$ can always be found and the analysis works in the same way as shown.
    \item General hypergraphs with planar hyperedges: after the previous point, it is clear that we can construct and decompose $\Omega^\epsilon$ in the same fashion for any finite hypergraph, provided $\epsilon$ sufficiently small.
    \item Hypergraphs with smooth nonplanar hyperedges: in this case, the reference mappings become nonlinear mappings and many aspects become technically much more involved. The cross section $\epsilon\omega$ of the node $\Node^\epsilon$ may depend smoothly on the tangential coordinate, but it will always have positive diameter and will be bounded as long as the hyperedges are smooth manifolds. Again, there will be an upper bound for $\epsilon$, but the limit properties will not be affected.
    \item Higher dimensional hypergraphs and graphs: here we end up in a situation with new reference domains
    \begin{align*}
 \Eref &:= (0,L)^{\locDim}  \times \left(-\frac12,\frac12\right)^{\globDim-\locDim},\\
 \Nref &:= (0,L)^{\locDim-1} \times \omega
\end{align*}
with $|\epsilon \omega| \sim \epsilon^{\globDim+1-\locDim}$.
\item More general assumptions on $\kappa^{\epsilon}$ are possible. For example   continuity of $\kappa^{\epsilon}$ on $\Edge_i^{\epsilon}$ and $\Node$ as in the definition of $f^{\epsilon}$  is enough (jumps between the different compartments are valid). Another possible choice is to construct $\kappa^{\epsilon}$ from a $L^{\infty}$-function on $\Edge_i$ and $\Node$, constantly extended in normal direction with respect to $\Edge_i$ resp. $\Node$.
\end{enumerate}

We conclude this section with some observations about the singular limit model and the 3D model problem:

First, we observe that in the limit problem the angles under which two or more hyperedges meet have become irrelevant, while these angles certainly have been relevant for the 3D model problem. Note that the definition of $\omega$ needs that there are no angles of zero degrees and if there are small angles, the shape of $\omega$ will compensate this drawback. Since we take $\epsilon \searrow 0$, this shape as well as the volume of $\omega$ is does not affect the limit problem. This effect implies that the $\epsilon$-limit of solutions will be approached slower when angles become small.

Second, nodal sources $g$ correspond to ``strong" sources $f^\epsilon_\Node$: Since the measures (and thereby the effects) of hypernodes are scaled by $\epsilon^2$ and vanish more quickly than those of hyperedges, sources on hypernodes are only relevant if they are of order of $\epsilon^{-1}$, cf.\ Remark \ref{REM:scale_f}. That is, analytically sources even become stronger during the limiting process.